\documentclass[12pt,reqno]{article}
\usepackage[letterpaper,margin=1.75cm]{geometry}
\usepackage{latexsym, amsfonts, amsthm,amssymb,amscd,amsmath,makeidx}
\usepackage{enumerate}
\usepackage[normalem]{ulem}
\usepackage{exscale}
\usepackage{overpic} 
\usepackage{color} 
\usepackage{graphicx}
\usepackage{overpic}  
\usepackage{bm}
\usepackage{bbm}
\usepackage{fancyhdr} 
\usepackage{mathrsfs}
\usepackage{wrapfig}
\usepackage{hyperref}

\hypersetup{
    colorlinks,
    citecolor=black,
    filecolor=black,
    linkcolor=blue,
    urlcolor=black
}

\definecolor{DarkGreen}{rgb}{0.2,0.6,0.2}

\def\eps{\varepsilon}

\def\lra{\longrightarrow}

\def\ua{\uparrow}
\def\da{\downarrow}

\def\wt{\widetilde}

\def\ignore#1{}

\def\bR{{\mathbb R}}

\def\bC{\mathbb C}
\def\bZ{\mathbb Z}
\def\bN{\mathbb N}

\def\bD{{\mathbb D}}

\def\cN{{\mathscr N}}

\def\cS{{\mathscr S}}

\def\cM{{\mathscr M}}

\def\cZ{{\mathscr Z}}

\def\cR{{\mathscr R}}
\def\cT{{ T}}

\newtheorem{theorem}{Theorem}[section]
\newtheorem{proposition}[theorem]{Proposition}
\newtheorem{lemma}[theorem]{Lemma}
\newtheorem{corollary}[theorem]{Corollary}

\theoremstyle{definition}
\newtheorem{definition}[theorem]{Definition}
\newtheorem{example}[theorem]{Example}

\newtheorem{remark}[theorem]{Remark}

\def\lra{\longrightarrow}
\def\eps{\varepsilon}
\def\<{\langle}
\def\>{\rangle}
\def\wt#1{\widetilde{#1}}

 \def\lra{\longrightarrow}

 \def\ua{\uparrow}
 \def\da{\downarrow}

 \def\wt{\widetilde}

 \def\argmax{\mathop{\hbox{\rm arg\,max}}}

\title{Step roots of Littlewood polynomials and\\  the extrema of  functions in the Takagi class}
\author{Xiyue Han$^*$ and Alexander Schied\thanks{
 University of Waterloo, 200 University Ave W, Waterloo, Ontario, N2L 3G1, Canada. E-Mails: {\tt xiyue.han@uwaterloo.ca, alex.schied@gmail.com}.\hfill\break
The authors gratefully acknowledge support from the
 Natural Sciences and Engineering Research Council of Canada through grant RGPIN-2017-04054.\hfill\break
This paper is based on the first author's thesis \cite{HanMMathThesis}, where its results were published in preliminary form. }}

\begin{document}
\date{\normalsize January 5, 2020} 
\maketitle
\begin{abstract}We give a new approach to characterizing and computing the set of global maximizers and minimizers of the functions in the Takagi class and, in particular, of the Takagi--Landsberg functions. The latter form a family of fractal functions $f_\alpha:[0,1]\to\bR$ parameterized by $\alpha\in(-2,2)$. 
We show that $f_\alpha$ has a unique maximizer in $[0,1/2]$ if and only if there does not exist a Littlewood polynomial that has $\alpha$ as a certain type of root, called step root.  Our general results lead to explicit and closed-form expressions for  the maxima of the Takagi--Landsberg functions with $\alpha\in(-2,1/2]\cup(1,2)$. 
For $(1/2,1]$, we show that the step roots are dense in that interval. If $\alpha\in (1/2,1]$ is a step root, then the set of maximizers of $f_\alpha$ is an explicitly given perfect set with Hausdorff dimension $1/(n+1)$, where $n$  is the degree  of the minimal Littlewood polynomial that has $\alpha$ as its step root. In the same way, we determine explicitly the minima of all Takagi--Landsberg functions. As a corollary, we show that the closure of the set of all real roots of all Littlewood polynomials is equal to $[-2,-1/2]\cup[1/2,2]$.\end{abstract}



\section{Introduction}

Rough paths calculus~\cite{FrizHairer} and the recent extension~\cite{ContPerkowski}  of F\"ollmer's pathwise It\^o calculus~\cite{FoellmerIto} provide  means of dealing with  rough trajectories that are not ultimately based on Gaussian processes such as fractional Brownian motion. As observed, e.g., in~\cite{GubinelliImkellerPerkowski}, such a pathwise calculus becomes particularly transparent when expressed in terms of the Faber--Schauder expansions of the integrands. When looking for the Faber--Schauder expansions of trajectories that are suitable pathwise integrators and  that have   \lq\lq roughness" specified in terms of a  given Hurst parameter, one is naturally led~\cite{MishuraSchied2} to certain extensions of a well-studied class of fractal functions, the Takagi--Landsberg functions.  These functions are defined as 
$$
f_\alpha(t):=\sum_{m=0}^\infty\frac{\alpha^m}{2^m}\phi(2^m t),\qquad 0\le t\le1,
$$
where $\alpha\in(-2,2)$ is a real parameter and 
$$
\phi(t):=\min_{z\in\bZ}|t-z|,\qquad t\in\bR,
$$
 is  the tent map,
If such functions are used to describe rough phenomena in applications, it is a natural question to analyze the range of these functions, i.e., to  determine the extrema of the Takagi--Landsberg functions.

While the preceding paragraph describes our original motivation for the research presented in this paper,  determining the maximum of generalized Takagi functions is also of intrinsic mathematical interest and attracted several authors in the past. The first contribution was by Kahane~\cite{Kahane}, who found the maximum and the set of maximizers of the classical Takagi function, which corresponds to $\alpha=1$. This result was later rediscovered in~\cite{Martynov} and subsequently extended in~\cite{Baba} to certain van der Waerden functions. 
Tabor and Tabor~\cite{TaborTabor} computed the maximum value of the Takagi--Landsberg function for those parameters $\alpha_n\in(1/2,1]$ that are characterized by $1-\alpha_n-\cdots-\alpha_n^{n}=0$ for $n\in\bN$. Galkin and Galkina~\cite{GalkinGalkina} proved that the maximum for $\alpha\in[-1,1/2]$ is attained at $t=1/2$. In the interval $(1,2)$, the case $\alpha=\sqrt2$ is special, as it corresponds to the Hurst parameter $H=1/2$. 
The corresponding maximum can be deduced from~\cite[Lemma 5]{Galkina} and was given independently in~\cite{GalkinGalkina} and~\cite{SchiedTakagi}.  Mishura and Schied~\cite{MishuraSchied2} added uniqueness to  the results from~\cite{GalkinGalkina,SchiedTakagi}  and extended them to all $\alpha\in(1,2)$.  The various contributions from~\cite{Kahane,TaborTabor,GalkinGalkina,MishuraSchied2} are illustrated in Figure~\ref{hist fig}, which shows the largest maximizer of the Takagi--Landsberg function $f_\alpha$ as a function of $\alpha$. 
From Figure~\ref{hist fig}, it is apparent that the most interesting cases are $\alpha\in(-2,-1)$ and $\alpha\in(1/2,1]$, which are also the ones about which nothing was  known beyond the special parameters considered in~\cite{Kahane} and~\cite{TaborTabor}.

In this paper, we present a completely new approach to the computation of the maximizers of the functions $f_\alpha$. This approach  works simultaneously for all parameters $\alpha\in(-2,2)$. It even extends to the entire Takagi class, which was introduced by Hata and Yamaguti~\cite{HataYamaguti} and is formed by all functions of the form 
\begin{equation*}
f(t):=\sum_{m=0}^\infty c_m\phi(2^mt),\qquad t\in[0,1],
\end{equation*}
where $(c_m)_{m\in\bN}$ is an absolutely summable  sequence. An example is the choice $c_m=2^{-m}\eps_m$, where $(\eps_m)_{m\in\bN}$ is an i.i.d.~sequence of $\{-1,+1\}$-valued Bernoulli random variables.  For this example, the distribution of the maximum was  studied by Allaart~\cite{AllaartRandomMax}.
Our approach works for arbitrary sequences $(c_m)_{m\in\bN}$  and provides a recursive characterization of the binary expansions of all maximizers and minimizers. This characterization is called  the \emph{step condition}. It yields a simple method to compute the smallest and largest maximizers and minimizers of $f$ with arbitrary precision. Moreover, it allows us to give exact statements on the cardinality of the set of   maximizers and minimizers of $f$. For the case of the Takagi--Landsberg functions, we find that, for $\alpha\in(-2,-1)$, the function $f_\alpha$ has either two or four maximizers, and we provide their exact values and the maximum values of $f_\alpha$ in closed form. For $\alpha\in[-1,1/2]$, the function $f_\alpha$ has a unique maximizer at $t=1/2$, and for $\alpha\in(1,2)$ there are exactly two maximizers at $t=1/3$ and $t=2/3$. The case $\alpha\in(1/2,1]$ is the most interesting. It will be discussed below.

In general, we show that non-uniqueness of maximizers in $[0,1/2]$ occurs if and only if there exists a Littlewood polynomial $P$ such that the parameter $\alpha$ is a special root of $P$,  called a \emph{step root}. The step roots also coincide with the discontinuities of the functions that assign to each $\alpha\in(-2,2)$ the respective smallest and largest maximizer of $f_\alpha$  in $[0,1/2]$. We show that the polynomials  $1-x-\cdots-x^{2n}$ are the only Littlewood polynomials with negative step roots, which in turn all belong to the interval $(-2,-1)$. They correspond exactly to the  jumps in $(-2,-1)$ of  the function  in Figure~\ref{hist fig}. While there are no step roots in $[-1,1/2]\cup(1,2)$, we show that the step roots lie dense in $(1/2,1]$. Moreover, if $n$ is  the smallest degree of a Littlewood polynomial that has $\alpha\in(1/2,1]$ as a step root, then the set of maximizers of $f_\alpha$ is a perfect set of Hausdorff dimension $1/(n+1)$, and the binary expansions of all maximizers are given in explicit form in terms of the coefficients of the corresponding Littlewood polynomial. As a corollary, we show that the closure of the set of all real roots of all Littlewood polynomials is equal to $[-2,-1/2]\cup[1/2,2]$.

This paper is organized as follows. In Section~\ref{Takagi class section}, we present general results for functions of the form \eqref{Takagi class eq}. In Section~\ref{TL section}, we discuss the particular case of the Takagi--Landsberg functions. The global maxima of $f_\alpha$ for the cases in which $\alpha$ belongs to the intervals $(-2,-1)$, $[-1,1/2]$, $(1/2,1]$, and $(1,2)$ are analyzed separately in the respective Subsections~\ref{alpha<1 results section},
~\ref{alpha in[-1,1/2] Section}, 
\ref{(1/2,1] section}, and 
\ref{alpha>1 results section}. We also discuss the global minima of $f_\alpha$ in Subsection~\ref{minima section}. As explained above, the maxima of the Takagi--Landsberg functions correspond to step roots of the Littlewood polynomials. Our results  yield  corollaries on the locations of such step roots and on the closure of the set of all real roots of the Littlewood polynomials. These corollaries are stated and proved in Section~\ref{step roots section}. The proofs of the results from Sections~\ref{Takagi class section} and~\ref{TL section} are deferred to the respective Sections~\ref{Takagi class proofs}
and~\ref{TL proofs}.

\begin{figure}
\begin{center}
\begin{overpic}[width=10cm]{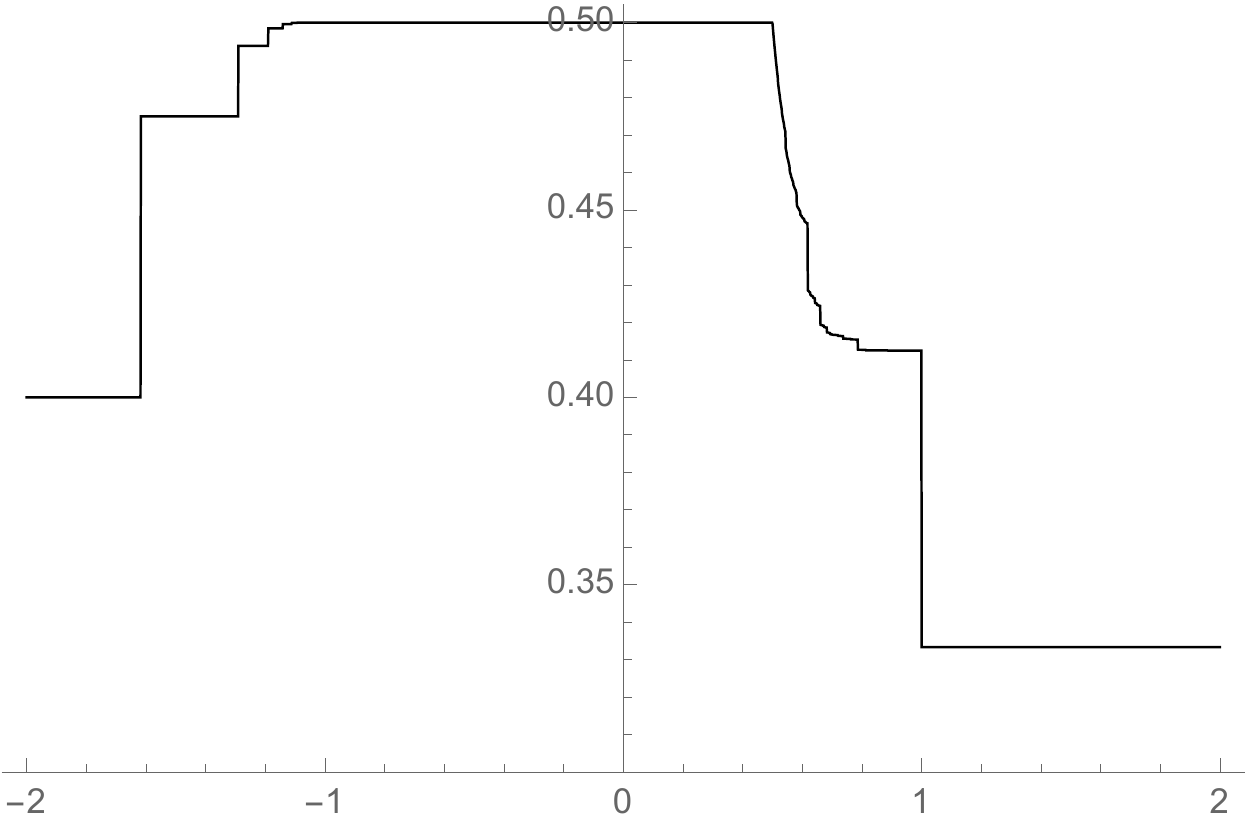}
\put(25,66){\footnotesize Galkin \& Galkina~\cite{GalkinGalkina}}
\put(74.5,16){\footnotesize Mishura \& Schied~\cite{MishuraSchied2}}
\put(74,39){\footnotesize Kahane~\cite{Kahane}}
\put(70,48){\footnotesize Tabor \& Tabor~\cite{TaborTabor} for}
\put(70,45.2){\footnotesize$1-\alpha-\alpha^2-\cdots-\alpha^n=0$}
\put(69,48){\vector(-1,-2){3}}
\put(69,48){\vector(-1,0){3}}
\put(69,48){\vector(-1,2){4.5}}
\end{overpic}\\
\caption{Maximizer of $t\mapsto f_\alpha(t)$ in $[0,1/2]$ as a function of $\alpha\in(-2,2)$.}\label{hist fig}
\end{center}

\end{figure}
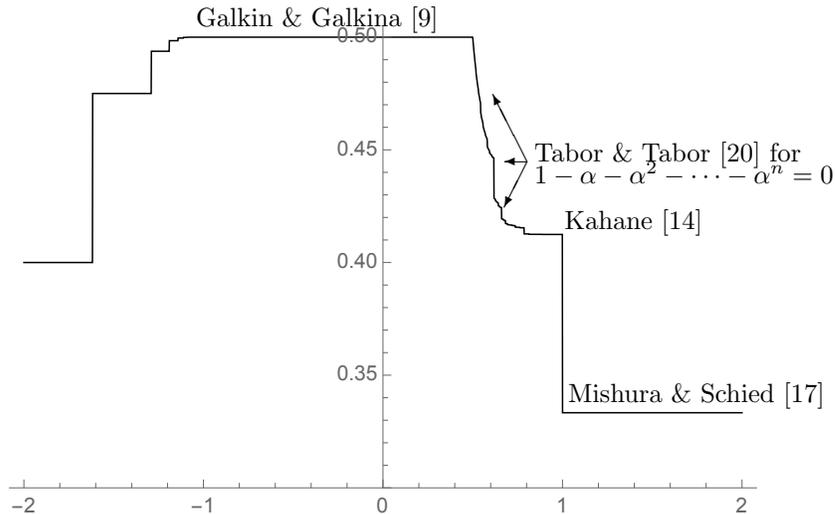

\section{Maxima of functions in the Takagi class}\label{Takagi class section}

The \emph{Takagi class} was introduced in~\cite{HataYamaguti}. It consists of the functions of the form
\begin{equation}\label{Takagi class eq}
f(t):=\sum_{m=0}^\infty c_m\phi(2^mt),\qquad t\in[0,1],
\end{equation}
where $\bm c=(c_m)_{m\in\bN}$ is a sequence in the space $\ell^1$ of absolutely summable sequences and 
$$
\phi(t):=\min_{z\in\bZ}|t-z|,\qquad t\in\bR,
$$
 is  the tent map. Under this assumption, the series in \eqref{Takagi class eq} converges uniformly in $t$, so that $f$ is a continuous function. The sequence $\bm c\in\ell^1$ will be fixed throughout this section.

For any $\{-1,+1\}$-valued sequence $\bm\rho=(\rho_m)_{m\in\bN_0}$, we let 
\begin{equation}\label{Psi eq}
\cT(\bm\rho)=\sum_{n=0}^\infty2^{-(n+2)}(1-\rho_n)\in[0,1]. 
\end{equation}
Then $\eps_n:=\frac12(1-\rho_n)$ will be the digits of a binary expansion of $t:=\cT(\bm\rho)$. We will call $\bm\rho$ a \emph{Rademacher expansion} of $t$. Clearly, the Rademacher expansion  is  unique unless $t$ is a dyadic rational number in $(0,1)$. Otherwise, $t$ will admit two distinct Rademacher expansions. The one with infinitely many occurrences of the digit $+1$ will be called the \emph{standard Rademacher expansion}. It can be obtained through the Rademacher functions, which are given by 
$
r_n(t):=(-1)^{\lfloor 2^{n+1} t\rfloor}$.
The following simple lemma illustrates the significance of the Rademacher expansion for the analysis of the function $f$.

\begin{lemma}\label{Rademacher f lemma}Let $\bm\rho=(\rho_m)_{m\in\bN_0}$ be a Rademacher expansion of $t\in[0,1]$. Then
\begin{equation*}
f(t)=\frac14\sum_{m=0}^\infty c_m\bigg(1-\sum_{k=1}^\infty 2^{-k}\rho_m\rho_{m+k}\bigg).
\end{equation*}
\end{lemma}

The following concept is the key to our analysis of the maxima of the function $f$.

\begin{definition}\label{gen step condition def}We will say that a $\{-1,+1\}$-valued sequence $(\rho_m)_{m\in\bN_0}$   satisfies the \emph{step condition} if 
$$\rho_{n}\sum_{m=0}^{n-1}2^mc_m\rho_m\le0\qquad\text{for all $n\in\bN$.}
$$
\end{definition}

Now we can state our first main result on the set
of 	 maximizers of $f$.

\begin{theorem}\label{general Takagi step cond thm}
For $t\in[0,1]$, the following conditions are equivalent.
\begin{enumerate}
\item $t$ is a maximizer of $f$;
\item every Rademacher expansion of $t$ satisfies the step condition; 
\item there exists a Rademacher expansion of $t$ that satisfies the step condition. 
\end{enumerate}
\end{theorem}


Theorem~\ref{general Takagi step cond thm} provides a way to construct maximizers of $f$. More precisely, we define  recursively the following pair of sequences $\bm\rho^{\flat}$ and $\bm\rho^{\sharp}$. We let $\rho^{\flat}_0=\rho^{\sharp}_0=1$ and, for $n\in\bN$,
\begin{equation}\label{rho+- eq}
\rho_n^\flat=\begin{cases}+1&\text{if $\sum_{m=0}^{n-1}2^mc_m\rho^\flat_m<0$,}\\
-1&\text{otherwise,}
\end{cases}\qquad \qquad 
\rho_n^\sharp=\begin{cases}+1&\text{if $\sum_{m=0}^{n-1}2^mc_m\rho^\sharp_m\le0$,}\\
-1&\text{otherwise.}
\end{cases}
\end{equation}

\begin{corollary}\label{smallest largest cor}With the above notation, $\cT(\bm\rho^\flat)$ is the largest and $\cT(\bm\rho^\sharp)$ is the smallest maximizer of $f$ in $[0,1/2]$.
\end{corollary}

\begin{remark}\label{minimizer rem}By switching the signs in the sequence $(c_n)_{n\in\bN_0}$, we get analogous results for the minima of the function $f$. Specifically, if we define sequences  $\bm\lambda^\flat$ and $\bm\lambda^\sharp$ by $\lambda^\flat_0=\lambda^\sharp_0=+1$ and 
\begin{equation*}
\lambda_n^\flat=\begin{cases}+1&\text{if $\sum_{m=0}^{n-1}2^mc_m\lambda^\flat_m>0$,}\\
-1&\text{otherwise,}
\end{cases}\qquad \qquad 
\lambda_n^\sharp=\begin{cases}+1&\text{if $\sum_{m=0}^{n-1}2^mc_m\lambda^\sharp_m\ge0$,}\\
-1&\text{otherwise,}
\end{cases}
\end{equation*}
then $\cT(\bm\lambda^\flat)$ is the largest and  $\cT(\bm\lambda^\sharp)$ is the smallest minimizer of $f$ in $[0,1/2]$. 
\end{remark}

The following corollary and its short proof illustrate the power of our method.

\begin{corollary}\label{general Takagi pos Cor}We have $f(t)\ge0$ for all $t\in[0,1]$, if and only if $\sum_{m=0}^{n}2^mc_m\ge0$ for all $n\ge0$.
\end{corollary}

\begin{proof}We have $f\ge0$ if and only if $t=0$ is the smallest minimizer of $f$. By Remark~\ref{minimizer rem}, this is equivalent to $\lambda_n^\sharp=+1$ for all $n$.
\end{proof}

Our method also allows to determine the cardinality of the set of maximizers of $f$. This is done in the following proposition.

\begin{proposition}\label{Takagi class number of max prop}
For $\bm\rho^\sharp$ as in \eqref{rho+- eq}, let
$$\cZ:=\Big\{n\in\bN_0\,\Big|\,\sum_{m=0}^n2^mc_m\rho_m^\sharp=0\Big\}.
$$
Then the number of $\{-1,+1\}$-valued sequences $\bm\rho$ that satisfy the step condition and $\rho_0=+1$ is $2^{|\cZ|}$ (where $2^{\aleph_0}$ denotes as usual the cardinality of the continuum). In particular, the number of maximizers of $f$ in $[0,1/2]$ is $2^{|\cZ|}$ less the number of maximizers in $(0,1/2)$ that are dyadic rationals.
\end{proposition}

\begin{example}\label{squaredex} Consider the function $f$ with $c_m=1/(m+1)^2$, which was considered in~\cite{HataYamaguti}. We claim that it
	has exactly two maximizers at $t={11}/{24}$ and $t={13}/{24}$. See Figure~\ref{squaredex fig}
 for an illustration. To prove our claim, we need to identify the sequence $\bm\rho^\sharp$ and show that the sums in \eqref{rho+- eq} never vanish.  A short computation yields  that $\rho^\sharp_0=1$, $\rho^\sharp_1=-1=\rho^\flat_1$, and $\rho^\sharp_2=-1=\rho^\flat_2$. To simplify the notation, we let $\bm\rho:=\bm\rho^\sharp$ and define
	$$R_n:=\sum_{m=0}^n\frac{2^m}{(m+1)^2}\rho_m.
	$$
	Next, we prove by induction on $n$ that for $n\ge2$,\begin{align}
	\rho_{2n-1}=-1\quad&\text{and}\quad \rho_{2n}=+1,\label{square_example_claim 1}
\\
	-\frac{2^{2n}}{(2n+1)^2} <R_{2n-1} < 0\quad&\text{and}\quad	0 < R_{2n} < \frac{2^{2n+1}}{(2n+2)^2}.\label{square_example_claim 2}
	\end{align} 
	To establish the case $n=2$, note first that $R_2=1/{18}$ and hence  $\rho_3=-1$. It follows that $R_3=1/18-8/16=-4/9$. This gives in turn that $\rho_4=+1$ and $R_4=-4/9+16/25=44/225$. This establishes \eqref{square_example_claim 1} and \eqref{square_example_claim 2} for $n=2$. Now suppose that our claims have been established for all $k$ with $2\le k\le n$. Then the second inequality in \eqref{square_example_claim 2} yields $\rho_{2n+1}=-1$ and in turn
	$$R_{2n+1}=R_{2n}-\frac{2^{2n+1}}{(2n+2)^2}>-\frac{2^{2n+1}}{(2n+2)^2}>-\frac{2^{2n+2}}{(2n+3)^2}\quad\text{and}\quad R_{2n+1}=R_{2n}-\frac{2^{2n+1}}{(2n+2)^2}<0.
	$$
	This yields $\rho_{2n+2}=+1$, from which we get as above that 
$$0<R_{2n+2}=R_{2n+1}+\frac{2^{2n+2}}{(2n+3)^2}<\frac{2^{2n+2}}{(2n+3)^2}<\frac{2^{2n+3}}{(2n+4)^2}.
$$
This proves our claims.  Furthermore, \eqref{Psi eq} yields that the unique maximizer in $[0,1/2]$ is given by
		\begin{equation*}
			T(\bm\rho) = \frac{1}{4} + \frac{1}{8} + \frac{1}{16}\sum_{n = 0}^{\infty} \frac{1}{4^n} = \frac{11}{24}.
		\end{equation*} 
\end{example}
\begin{figure}[h]
\begin{center}
   \includegraphics[width=0.4\textwidth]{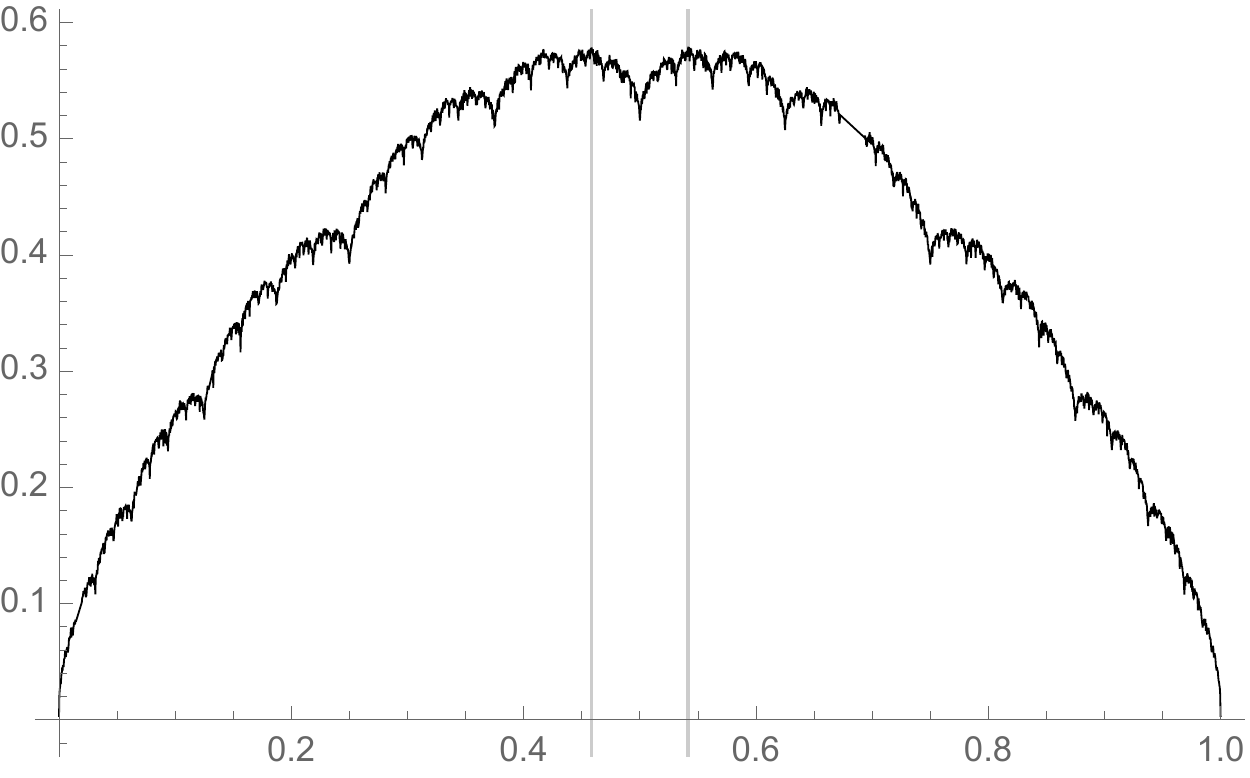}
   \caption{The function with $c_m=1/(m+1)^2$ analyzed in Example~\ref{squaredex}. The vertical lines correspond to the two maxima at ${11}/{24}$ and at ${13}/{24}$.}\label{squaredex fig}
\end{center}
\end{figure}

\section{Global extrema of the Takagi--Landsberg functions}\label{TL section}

	 The Takagi--Landsberg function with parameter $\alpha\in(-2,2)$ is given by 
	 \begin{equation}\label{Takagi-landsberg alpha parameterization eq}
f_\alpha(t):=\sum_{m=0}^\infty\frac{\alpha^m}{2^m}\phi(2^mt),\qquad t\in[0,1].
	 \end{equation}
	In the case $\alpha=1$, the function $f_1$ is the classical Takagi function, which was first  introduced by Takagi~\cite{Takagi} and  later rediscovered many times; see, e.g., the surveys~\cite{AllaartKawamura} and~\cite{Lagarias}. The class of functions $f_\alpha$ with $-2<\alpha<2$ is sometimes also called the exponential Takagi class. See Figure~\ref{TL fig} for an illustration.
\begin{figure}[h]
	\begin{center}
\begin{minipage}[b]{6cm}
\begin{overpic}[width=6cm]{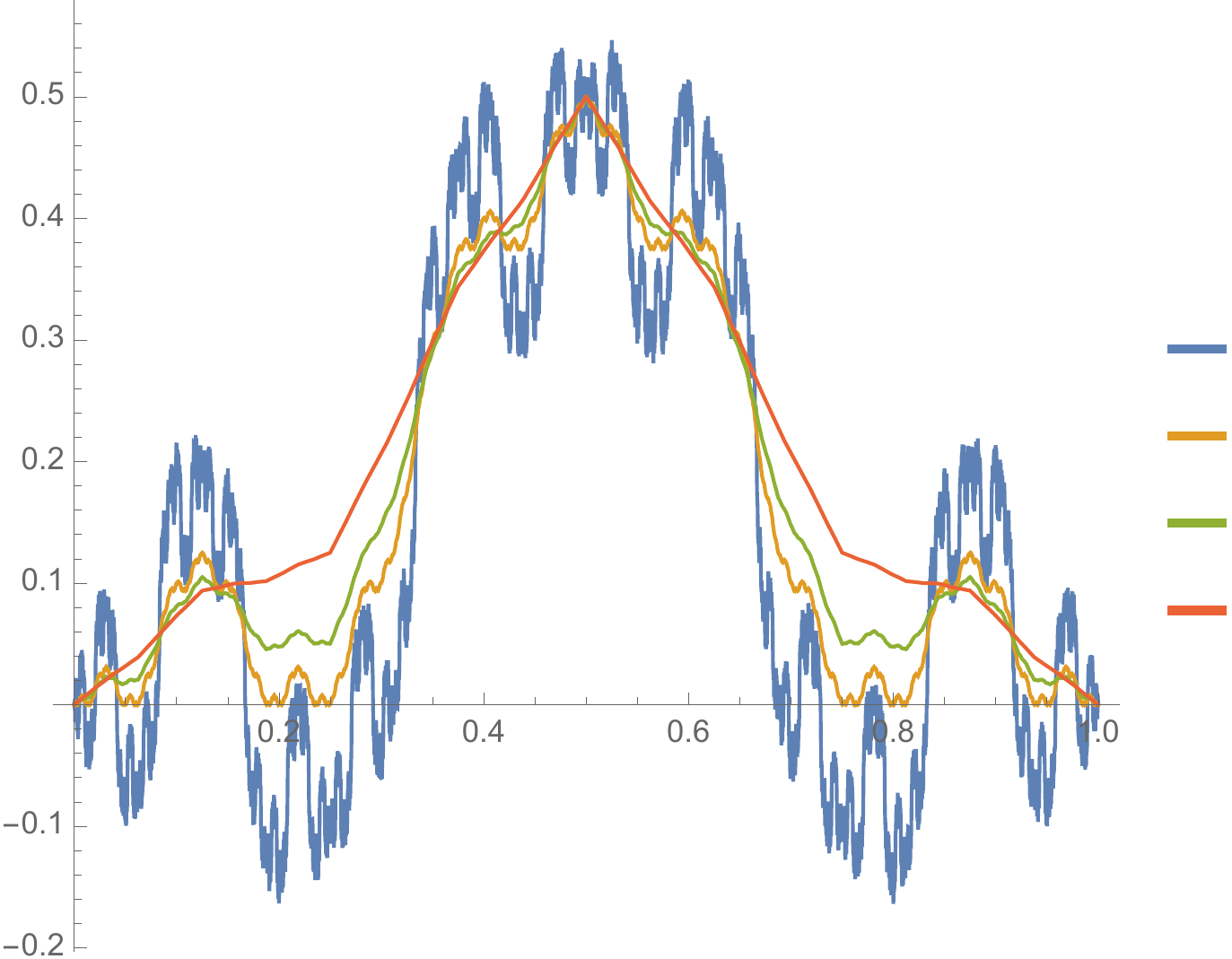}
\put(101,48){\footnotesize$\sqrt2$}
\put(101,41){\footnotesize$1$}
\put(101,34){\footnotesize$4/5$}
\put(101,27){\footnotesize$1/2$}
\end{overpic}
\end{minipage}\qquad\ \
\begin{minipage}[b]{6cm}
\begin{overpic}[width=5.2cm]{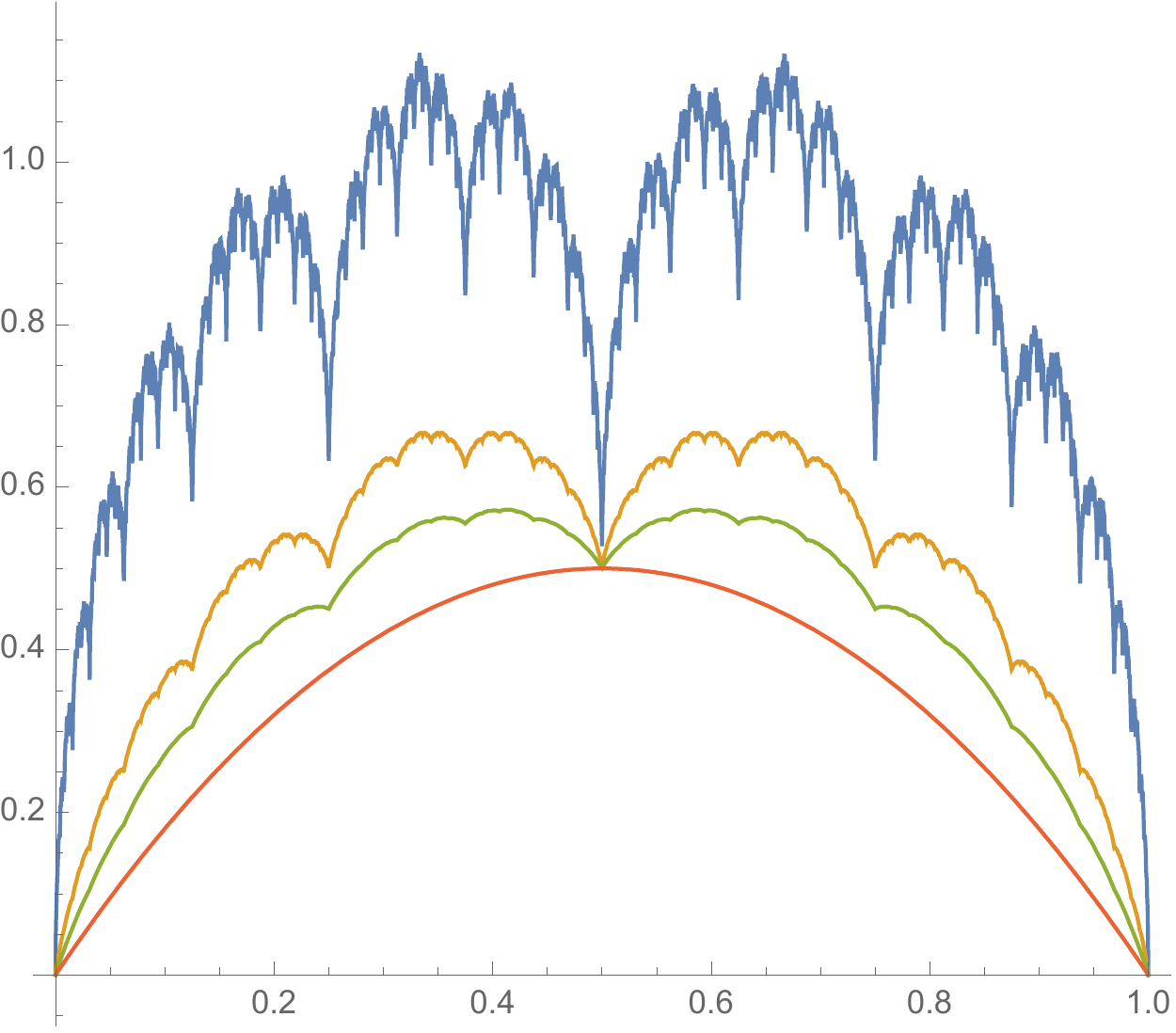}
\end{overpic}
\end{minipage}
\end{center}
\caption{Takagi--Landsberg functions $f_{-\alpha}$ (left) and $f_\alpha$ (right) for four different values of $\alpha$.}\label{TL fig}
\end{figure}

	By letting $c_m:=\alpha^m2^{-m}$, we see that the results from Section~\ref{Takagi class section} apply to the function $f_\alpha$. In particular,  	Theorem~\ref{general Takagi step cond thm}
 characterizes the maximizers of $f_\alpha$ in terms of a step condition satisfied by their Rademacher expansions. 
  Let us restate the corresponding Definition~\ref{gen step condition def} in our present situation.
 
 \begin{definition}\label{exponential Takagi step cond def}Let $\alpha\in(-2,2)$. A $\{-1,+1\}$-valued sequence $(\rho_m)_{m\in\bN_0}$   satisfies the \emph{step condition for $\alpha$} if 
$$\rho_{n}\sum_{m=0}^{n-1}\alpha^m\rho_m\le0\qquad\text{for all $n\in\bN$.}
$$
 \end{definition}

As in \eqref{rho+- eq}, we define  recursively the following pair of sequences $\bm\rho^{\flat}(\alpha)$ and $\bm\rho^{\sharp}(\alpha)$. We let $\rho^{\flat}_0(\alpha)=\rho^{\sharp}_0(\alpha)=1$ and, for $n\in\bN$,
\begin{equation}\label{rho+- eq bis}
\rho_n^\flat(\alpha)=\begin{cases}+1&\text{if $\sum_{m=0}^{n-1}\alpha^m\rho^\flat_m(\alpha)<0$,}\\
-1&\text{otherwise,}
\end{cases}\qquad \qquad 
\rho_n^\sharp(\alpha)=\begin{cases}+1&\text{if $\sum_{m=0}^{n-1}\alpha^m\rho^\sharp_m(\alpha)\le0$,}\\
-1&\text{otherwise.}
\end{cases}
\end{equation}
Then we define
\begin{equation*}
\tau^\flat(\alpha):=\cT(\bm\rho^\flat(\alpha))\qquad \text{and}\qquad \tau^\sharp(\alpha):=\cT(\bm\rho^\sharp(\alpha)),
\end{equation*}
where $\cT$ is as in \eqref{Psi eq}. It follows from Corollary~\ref{smallest largest cor} that $\tau^\flat(\alpha)$ is the largest and $\tau^\sharp(\alpha)$ is the smallest maximizer of $f_\alpha$ in $[0,1/2]$. 
We start with the following general result.

\begin{proposition}\label{continuity prop} For $\alpha\in(-2,2)$, the following conditions are equivalent.
\begin{enumerate}
\item\label{continuity prop (a)}  The function $f_\alpha$ has a unique maximizer in $[0,1/2]$.
\item\label{continuity prop (b)} $\tau^\flat(\alpha)=\tau^\sharp(\alpha)$.
\item\label{continuity prop (c)} There exists no $n\in\bN$ such that $\sum_{m=0}^n\alpha^m\rho_m^\sharp(\alpha)=0$.
\item\label{continuity prop (d)} The functions $\tau^\flat$ and $\tau^\sharp$ are continuous at $\alpha$.
\end{enumerate}
\end{proposition}

In the following subsections, we discuss the maximization of $f_\alpha$ for various regimes of $\alpha$.

\subsection{Global maxima for  $\alpha\in(-2,-1)$}\label{alpha<1 results section}

To the best of our knowledge, the case $\alpha\in(-2,-1)$ has not yet been discussed in the literature. Here, we give an explicit solution for both maximizers and maximum values in this regime. Before stating our corresponding result, we formulate the following elementary lemma.

\begin{lemma}\label{Littlewood neg root lemma}
For $n\in\bN$, the Littlewood polynomial $p_{2n}(x)=1-x-\cdots-x^{2n-1}-x^{2n}$ has a unique negative root $x_{n}$. Moreover,  the sequence $(x_{n})_{n\in\bN}$ is strictly increasing, belongs to $(-2,-1)$, and converges to $-1$ as $n\ua\infty$. 
\end{lemma}

Note that $-x_1= \frac12(1+\sqrt5)\approx 1.61803$ is the  golden ratio. Approximate numerical values for the next highest roots are $x_2\approx -1.29065$, $x_3\approx -1.19004$,  $x_4\approx -1.14118$, and $x_5\approx -1.11231$.

\begin{theorem} \label{alpha negative thm}Let $(x_{n})_{n\in\bN}$ be the sequence introduced in Lemma~\ref{Littlewood neg root lemma} and define $x_0:=-2$. Then, on $(x_n,x_{n+1})$,
the function $f_\alpha$ has exactly two maximizers in $[0,1]$, which are located at 
$$t_n:=\frac1{10}(5-4^{-n})\quad\text{and}\quad 1-t_n=\frac1{10}(5+4^{-n}).$$ 
If $\alpha=x_n$ for some $n\in\bN$, then $f_\alpha$ has exactly four maximizers in $[0,1]$, which are  located at $t_n$, $t_{n+1}$, $1-t_n$ and $1-t_{n+1}$. Moreover,
\begin{equation}\label{alpha negative thm minimum eq}
f_\alpha(t_n)=\frac{1}{10} (5-4^{-n})-\frac{4^{-n}}{10} \cdot\frac{ 3 \alpha^{2 n+3}+ \alpha^3-4 \alpha}{(1-\alpha)
   \left(\alpha^2-4\right)},
\end{equation}
and this is equal to the maximum value of $f_\alpha$ if $\alpha\in[x_n,x_{n+1}]$.
\end{theorem}

\begin{remark}\label{alpha negative rem}
It is easy to see that the right-hand side of \eqref{alpha negative thm minimum eq} is strictly larger than $1/2$ for $\alpha\in(-2,-1)$. Moreover, it tends to $+\infty$ for $\alpha\downarrow-2$ and to $1/2$ for $\alpha\ua-1$.
\end{remark}

\subsection{Global maxima for  $\alpha\in[-1,1/2]$}\label{alpha in[-1,1/2] Section}

Galkin and Galkina~\cite{GalkinGalkina} proved that for $\alpha\in[-1,1/2]$ the function $f_\alpha$ has a global maximum at $t=1/2$ with maximum value $f_\alpha(1/2)=1/2$. Here, we give a short proof of  this result by using our method and additionally establish the uniqueness of the maximizer. 

\begin{proposition}\label{alpha in -1 1/2 prop} For $\alpha\in[-1,1/2]$, the function $f_\alpha$ has the unique maximizer $t=1/2$ and the maximum value $f_\alpha(1/2)=1/2$.
\end{proposition}

\begin{proof}[Proof of Proposition~\ref{alpha in -1 1/2 prop}] Since obviously $f_\alpha(1/2)=1/2$ for all $\alpha$, the result will follow if we can establish that $\tau^\sharp(\alpha)=1/2$ for all $\alpha\in[-1,1/2]$. This is the case if $\bm\rho:=\bm\rho^\sharp(\alpha)$ satisfies $\rho_n=-1$ for all $n\ge1$. We prove this by induction on $n$. The case $n=1$ follows immediately from $\rho_0=1$ and \eqref{rho+- eq bis}. If $\rho_1=\cdots=\rho_{n-1}=-1$ has already been established, then 
$$\sum_{m=0}^{n-1}\alpha^m\rho_m=\frac{\alpha^n-2\alpha+1}{1-\alpha}.
$$
If the right-hand side is strictly positive, then we have $\rho_n=-1$. Positivity is obvious for $\alpha\in[-1,0]$ and for $\alpha=1/2$. For $\alpha\in(0,1/2)$, we can take the derivative of the numerator with respect to $\alpha$. This derivative is equal to $n\alpha^{n-1}-2$, which is strictly negative for $\alpha\in(0,1/2)$, because $n\alpha^{n-1}\le1$ for those $\alpha$. Since the numerator is strictly positive for $\alpha=1/2$, the result follows.
\end{proof}

\subsection{Global maxima for $\alpha\in(1/2,1]$}\label{(1/2,1] section}

This is the most interesting regime, as can already be seen from Figure~\ref{hist fig}.  Kahane~\cite{Kahane} showed that the maximum value of the classical Takagi function $f_1$ is $2/3$ and that the set of maximizers is equal to the set of all points in $[0,1]$ whose binary expansion satisfies $\eps_{2n}+\eps_{2n+1}=1$ for each $n\in\bN_0$. This is a perfect set of Hausdorff dimension $1/2$. For other values of $\alpha\in(1/2,1]$, we are only aware of the following result by 
Tabor and Tabor~\cite{TaborTabor}. They found the maximum value of $f_{\alpha_n}$, where $\alpha_n$ is the unique positive root of the Littlewood polynomial $1-x-x^2-\cdots -x^n$. This sequence satisfies $\alpha_1=1$ and $\alpha_n\da1/2$ as $n\ua\infty$. The maximum value of $f_{\alpha_n}$ is then given by $C(\alpha_n)$, where
\begin{equation}\label{TaborTabos max}
C(\alpha):=\frac{1}{2-2(2 \alpha -1)^{\frac{\log_2 \alpha-1}{\log_2 \alpha }}}.
\end{equation}
Tabor and Tabor~\cite{TaborTabor} observed numerically that  the maximum value of $f_\alpha$ typically \emph{differs} from $C(\alpha)$  for other values of $\alpha\in(1/2,1)$. In Example~\ref{TaborTabor counterex} we will investigate a specific choice of $\alpha$ for which $C(\alpha)$ is indeed different from the maximum value of $f_\alpha$. In Example~\ref{TaborTaborEx}, we will characterize the set of maximizers of $f_{\alpha_n}$, where  $\alpha_n$ is as above. 

We have seen in Sections~\ref{alpha<1 results section}
 and~\ref{alpha in[-1,1/2] Section} that for $\alpha\le1/2$ the function $f_\alpha$ has either two, or four maximizers in $[0,1]$. For $\alpha>1/2$ this situation changes. The following result shows that then $f_\alpha$ will have either two or uncountably many maximizers. Moreover, the result quoted in Section~\ref{alpha>1 results section}
 will imply that the latter case can only happen for $\alpha\in(1/2,1]$.
 
 \begin{theorem}\label{Cantor thm} For $\alpha>1/2$, we have the following dichotomy.
 \begin{enumerate}
 \item If $\sum_{m=0}^n\alpha^m\rho^\sharp_m\neq0$ for all $n$, then the function $f_\alpha$ has exactly  two maximizers in $[0,1]$. They are given by $\tau^\sharp(\alpha)$  and $1-\tau^\sharp(\alpha)$ and have $\bm\rho^\sharp$ and $-\bm\rho^\sharp$ as their Rademacher expansions. 
 \item Otherwise, let $n_0$ be the smallest $n$ such that $\sum_{m=0}^n\alpha^m\rho^\sharp_m=0$. Then the set of maximizers of $f_\alpha$ consists of all those $t\in[0,1]$ that have a Rademacher expansion consisting of successive blocks of the form $\rho^\sharp_0,\dots, \rho^\sharp_{n_0}$ or $(-\rho^\sharp_0),\dots,(- \rho^\sharp_{n_0})$. This is a perfect set of Hausdorff dimension $1/(n_0+1)$ and its $1/(n_0+1)$-Hausdorff measure is finite and strictly positive.
 \end{enumerate}
 \end{theorem}
 
  The preceding theorem yields the following corollary.

 \begin{corollary}\label{dyadic cor} For $\alpha>1/2$, the function $f_\alpha$ cannot have a maximizer that is a dyadic rational number. 
 \end{corollary}
 
Note that Theorem~\ref{alpha negative thm} implies that also for $\alpha<-1$ there are no dyadic rational maximizers. However, by Proposition~\ref{alpha in -1 1/2 prop}, the unique maximizer in case $-1\le \alpha\le1/2$ is $t=1/2$. 

Our next result shows in particular that there is no nonempty open interval in $(1/2,1]$ on which $\tau^\flat$ or $\tau^\sharp$ are constant. 
 
\begin{theorem}\label{nowhere flat thm}There is no nonempty open interval in $(1/2,1)$ on which the functions $\tau^\flat$ or $\tau^\sharp$ are continuous. \end{theorem}

\begin{example}\label{TaborTaborEx} Tabor and Tabor~\cite{TaborTabor}  found the maximum value of $f_{\alpha_n}$, where $\alpha_n$ is the unique positive root of the Littlewood polynomial $1-x-x^2-\cdots -x^n$. The case $n=1$, and in turn $\alpha_1=1$, corresponds to the classical Takagi function  as studied by Kahane~\cite{Kahane}. Here, we will now determine the corresponding sets of maximizers. It is clear that we must have $1-\sum_{k=1}^m\alpha_n^k>0$ for $m=1,\dots, n-1$. Hence, 
\begin{align*}
\bm\rho^\sharp(\alpha_n)&=(+1,\underbrace{-1,\dots,-1}_{\text{$n$ times}},+1,\underbrace{-1,\dots,-1}_{\text{$n$ times}},\dots),\\
\bm\rho^\flat(\alpha_n)&=(+1,\underbrace{-1,\dots,-1}_{\text{$n$ times}},-1,\underbrace{+1,\dots,+1}_{\text{$n$ times}},-1,\underbrace{+1,\dots,+1}_{\text{$n$ times}},\dots).
\end{align*}
Every maximizer in $[0,1]$ has a Rademacher expansion that is made up of successive blocks of  length $n+1$ taking the form $+1,-1,\dots,-1$ or $-1,+1\dots, +1$. This is a perfect set of Hausdorff dimension $1/(n+1)$. The smallest maximizer is given by
\begin{align*}
\tau^\sharp(\alpha_n)&=\sum_{m=0}^\infty\sum_{k=2}^{n+1}2^{-(m(n+1)+k)}=\Big(\frac12-2^{-(n+1)}\Big)\sum_{m=0}^\infty2^{-m(n+1)}=\frac{2^n-1}{2^{n+1}-1}.
\end{align*}
The largest maximizer in $[0,1/2]$ is
$$\tau^\flat(\alpha_n)=\frac12-2^{-(n+1)}+\sum_{m=1}^\infty 2^{-m(n+1)-1}=\frac{1}{2} \Big(1-2^{-n}+\frac{1}{2^{n+1}-1}\Big).
$$\end{example}

\begin{example}\label{TaborTabor counterex}Consider the choice
$$\alpha=\frac14\Big(1+\sqrt{13}-\sqrt{2\big(\sqrt{13}-1\big)}\Big)\approx 0.580692.$$
One checks that $1-\alpha-\alpha^2-\alpha^3+\alpha^4=0$ and that $1-\alpha-\alpha^2-\alpha^3<0$ and $1-\alpha-\alpha^2>0$ and $1-\alpha>0$. Therefore, 
\begin{align*}
\bm\rho^\sharp(\alpha)&=(+1,-1,-1,-1,+1,+1,-1,-1,-1,+1,+1,-1,-1,-1,+1,\dots)\\
\bm\rho^\flat(\alpha)&=(+1,-1,-1,-1,+1,-1,+1,+1,+1,-1,-1,+1,+1,+1,-1,\dots),
\end{align*}
and every maximizer in $[0,1]$ has a Rademacher expansion that consists  of successive blocks of the form $+1,-1,-1,-1,+1$ or $-1,+1,+1,+1,-1$. This is a Cantor-type set of Hausdorff dimension $1/5$. Furthermore,
\begin{align*}
\tau^\sharp(\alpha)&=\sum_{n=0}^\infty\big(0\cdot 2^{-(5n+1)}+2^{-(5n+2)}+ 2^{-(5n+3)}
+ 2^{-(5n+4)}+0\cdot 2^{-(5n+5)}\big)=\frac{14}{31}\approx 0.451613
\end{align*}
is the smallest maximizer, and
$$\tau^\flat(\alpha)=\frac7{16}+\sum_{n=1}^\infty \big(2^{-5n-1}+2^{-5n-5}\big)=\frac{451}{992}\approx 0.454637
$$
is the largest maximizer in $[0,1/2]$. To compute the maximum value, we can either use Lemma~\ref{Rademacher f lemma}, or we directly compute $f_\alpha(14/31)$ as follows. We note that $\phi(2^{5n+k}14/31)=b_k/31$, where $b_0=14$, $b_1=3$, $b_2=6$, $b_3=12$, and $b_4=7$. Thus,
\begin{align*}
f_\alpha(14/31)&=\frac1{31}\sum_{n=0}^\infty\Big(\frac\alpha2\Big)^{5n}\Big(14+3\frac\alpha2+6\Big(\frac\alpha2\Big)^2+12\Big(\frac\alpha2\Big)^3+7\Big(\frac\alpha2\Big)^4\Big)\\
&=\frac{39+3 \sqrt{13}-\sqrt{6 \left(25+7 \sqrt{13}\right)}}{56-2\sqrt{13}+4 \sqrt{7+2 \sqrt{13}}}\approx 0.508155,\end{align*}
where the second identity was obtained by using {\tt Mathematica 12.0}.
For the function in \eqref{TaborTabos max}, we get, however, $C(\alpha)\approx 0.508008$, which confirms the numerical observation from~\cite{TaborTabor} that $C(\cdot)$ may not yield correct maximum values if evaluated at arguments different from the positive roots of $1-x-x^2-\cdots -x^n$.
\end{example}

\subsection{Global maxima for $\alpha\in(1,2)$}\label{alpha>1 results section}

For $\alpha=\sqrt2$, it can be deduced from~\cite[Lemma 5]{Galkina} that $f_{\sqrt2}$ has  maxima  at $t=1/3$ and $t=2/3$ and maximum value $\frac13(2+\sqrt2)$. That lemma was later rediscovered by the second author in~\cite[Lemma 3.1]{SchiedTakagi}. The statement on the maxima of $f_{\sqrt2}$ was given independently in~\cite{GalkinGalkina} and~\cite{SchiedTakagi}. 
Mishura and Schied~\cite{MishuraSchied2} extended this subsequently to the following result, which we quote here for the sake of completeness. It is not difficult to prove it with our present method; see~\cite[Example 4.3.1]{HanMMathThesis}.

\begin{theorem}[\bfseries{Mishura and Schied~\cite{MishuraSchied2}}]\label{alpha in (1,2) thm} For $\alpha\in(1,2)$, the function $f_\alpha$ has exactly  two maximizers at $t=1/3$ and $t=2/3$ and its maximum value  is $(3(1-\alpha/2))^{-1}$.
\end{theorem}

\subsection{Global minima}\label{minima section}

In this section, we discuss the minima of the function $f_\alpha$.

\begin{theorem}\label{minima thm} For the global minima of the function $f_\alpha$, we have the following three cases.
\begin{enumerate}
\item\label{minima thm (a)} For $\alpha\in(-2,-1)$, the function $f_\alpha$ has a unique minimum in $[0,1/2]$, which is located at $t=1/5$. Moreover, the minimum value is
$$f_\alpha(1/5)=\frac{1+\alpha}{5(1-(\alpha/2)^2)}.
$$
\item\label{minima thm (b)} For $\alpha=-1$, the minimum value of $f_\alpha$ is equal to $0$, and the set of minimizers is equal to the set of all $t\in[0,1]$ that have a Rademacher expansion $\bm\rho$ with $\rho_{2n}=\rho_{2n+1}$ for $n\in\bN_0$. This is a perfect set of Hausdorff dimension $1/2$, and its $1/2$-dimensional Hausdorff measure is finite and strictly positive.
\item\label{minima thm (c)} For $\alpha\in(-1,2)$, the unique minimizer of $f_\alpha$ in $[0,1/2]$ is at $t=0$ and the minimum value is $f_\alpha(0)=0$.
\end{enumerate}
\end{theorem}

The preceding theorem and Remark~\ref{alpha negative rem} yield immediately  the following corollary.

 \begin{corollary}\label{nonpositive cor}The function $f_\alpha(t)$ is nonnegative for all $t\in[0,1]$ if and only if $\alpha\ge-1$. Moreover, there is no $\alpha\in(-2,2)$ such that $f_\alpha$ is nonpositive.  
 \end{corollary}

	The fact that $f_\alpha\ge0$ for $\alpha\ge-1$ can alternatively be deduced from an argument in the proof of~\cite[Theorem 4.1]{GalkinGalkina}.

\section{Real (step) roots of Littlewood polynomials}\label{step roots section}

In this section, we link our analysis of the maxima of the Takagi--Landsberg functions to certain real roots of the Littlewood polynomials. Recall that a Littlewood polynomial is a polynomial whose coefficients are all $-1$ or $+1$. By Corollary 3.3.1 of~\cite{BaradaranTaghavi}, the complex roots of any Littlewood polynomial must lie in the annulus $\{z\in\bC\,|\,1/2<|z|<2\}$. Hence, the real roots can only lie in $(-2,-1/2)\cup(1/2,2)$. Below, we will show in Corollary~\ref{dense cor} that the real roots are actually dense in that set. We start with the following simple lemma.

\begin{lemma}\label{rational root lemma}
The numbers $-1$ and $+1$ are the only  rational roots for Littlewood polynomials.
\end{lemma}
\begin{proof}
Assume $\alpha \in \mathbb{Q}$ is a rational root for some Littlewood polynomial $P_n(x)$. Then the monic polynomial $x - \alpha$ divides $P_n(x)$. The Gauss lemma yields that $x - \alpha \in \mathbb{Z}[x]$ and hence $\alpha \in \mathbb{Z}$. By the above-mentioned  Corollary 3.3.1 of~\cite{BaradaranTaghavi}, we get $|\alpha| = 1$. 
\end{proof}

\begin{definition}
For given $n\in\bN$, let 
$P_n(x)=\sum_{m=0}^n\rho_mx^m
$
be a Littlewood polynomial with coefficients $\rho_m\in\{-1,+1\}$. If $k\le n$, we write $P_k(x)=\sum_{m=0}^k\rho_mx^m$. A number $\alpha\in\bR$ is called a \emph{step root} of $P_n$ if $P_n(\alpha)=0$ and $\rho_{k+1}P_k(\alpha)\le0$ for $k=0,\dots, n-1$. \end{definition}

 The concept of a step root has the following significance for the maxima of the Takagi--Landsberg functions $f_\alpha$ defined in \eqref{Takagi-landsberg alpha parameterization eq}. 
 
\begin{corollary}\label{max unique step root cor}For $\alpha\in(-2,2)$, the following conditions are equivalent.
\begin{enumerate}
\item The function $f_\alpha$ has a unique maximizer in $[0,1/2]$.
\item There is no Littlewood polynomial that has $\alpha$ as its step root.
\end{enumerate}
\end{corollary}
\begin{proof}
The assertion follows immediately from Proposition~\ref{continuity prop} and Theorem~\ref{general Takagi step cond thm}.
\end{proof}

With our results on the maxima of the Takagi--Landsberg function, we thus get the following corollary on the locations of the step roots of the Littlewood polynomials.

\begin{corollary}\label{step root location cor}We have the following results.
\begin{enumerate}
\item The only Littlewood polynomials admitting negative step roots are of the form
$1-x-x^2-\cdots-x^{2n}$ for some $n\in\bN$
and the step roots are the numbers $x_n$ in Lemma~\ref{Littlewood neg root lemma}.
\item There are no step roots in $[-1,1/2]\cup(1,2)$.
\item The step roots are dense in $(1/2,1]$.
\end{enumerate}
\end{corollary}

\begin{proof}In view of Corollary~\ref{max unique step root cor}, (a) follows from Theorem~\ref{alpha negative thm}. Assertion (b) follows from Proposition~\ref{alpha in -1 1/2 prop}  and Theorem~\ref{alpha in (1,2) thm}. Part (c) follows from Theorem~\ref{nowhere flat thm}. 
\end{proof}

From part (c) of the preceding corollary, we obtain the following result, which identifies  $[-2,-1/2]\cup[1/2,2]$ as the closure of the set of all real roots of the Littlewood polynomials. Although the roots of the Littlewood polynomials have been well studied in the literature (see, e.g,~\cite{Borwein} and the references therein), we were unable to find the following result in the literature. In \cite[E1 on p.~72]{Borwein}, it is stated that an analogous result holds if the Littlewood polynomials are replaced by the larger set of all polynomials with coefficients in $\{-1,0,+1\}$. In the student thesis \cite{Vader},  determining the closure of the real roots of the Littlewood polynomials was classified as an open problem. The distribution of the positive roots and step roots of Littlewood polynomials is illustrated in Figure \ref{roots fig}.

\begin{corollary}\label{dense cor}Let $\cR$ denote the set of all real roots of the Littlewood polynomials. Then the closure of $\cR$ is given by 
	 $[-2,-1/2]\cup[1/2,2]$.
\end{corollary}
\begin{proof} We know from~\cite[Corollary 3.3.1]{BaradaranTaghavi} that $\cR\subset(-2,-1/2)\cup(1/2,2)$. Now denote by $\cS$ the set of all step roots of the Littlewood polynomials, so that $\cS\subset\cR$. Corollary~\ref{step root location cor} (c) yields that $[1/2,1]$ is contained in the closure of $\cS$, and hence also in the closure of $\cR$.
Next, note that if $\alpha$ is the root of a Littlewood polynomial, then so is $1/\alpha$. Indeed,  if $\alpha$ is a root of the Littlewood polynomial $P(x)$, then $\wt P(x):=x^nP(1/x)$ is also a Littlewood polynomial and satisfies $\wt P(1/\alpha)=\alpha^{-n}P(\alpha)=0$. Hence, $[1,2]$ is contained in the closure of $\cR$. Finally, for $\alpha\in\cR$, we clearly have also $-\alpha\in\cR$. This completes the proof.
\end{proof}

\begin{figure}
\begin{center}
\includegraphics[width=7.5cm]{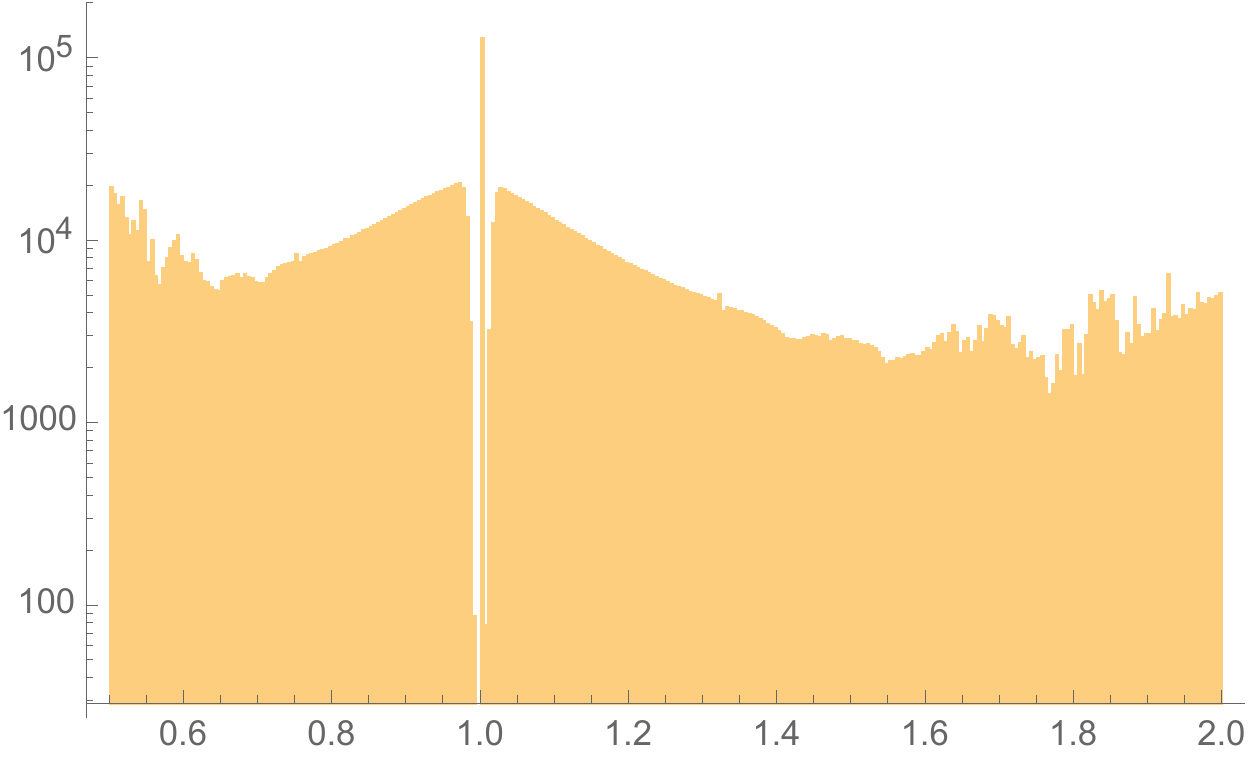}\qquad
\includegraphics[width=7.5cm]{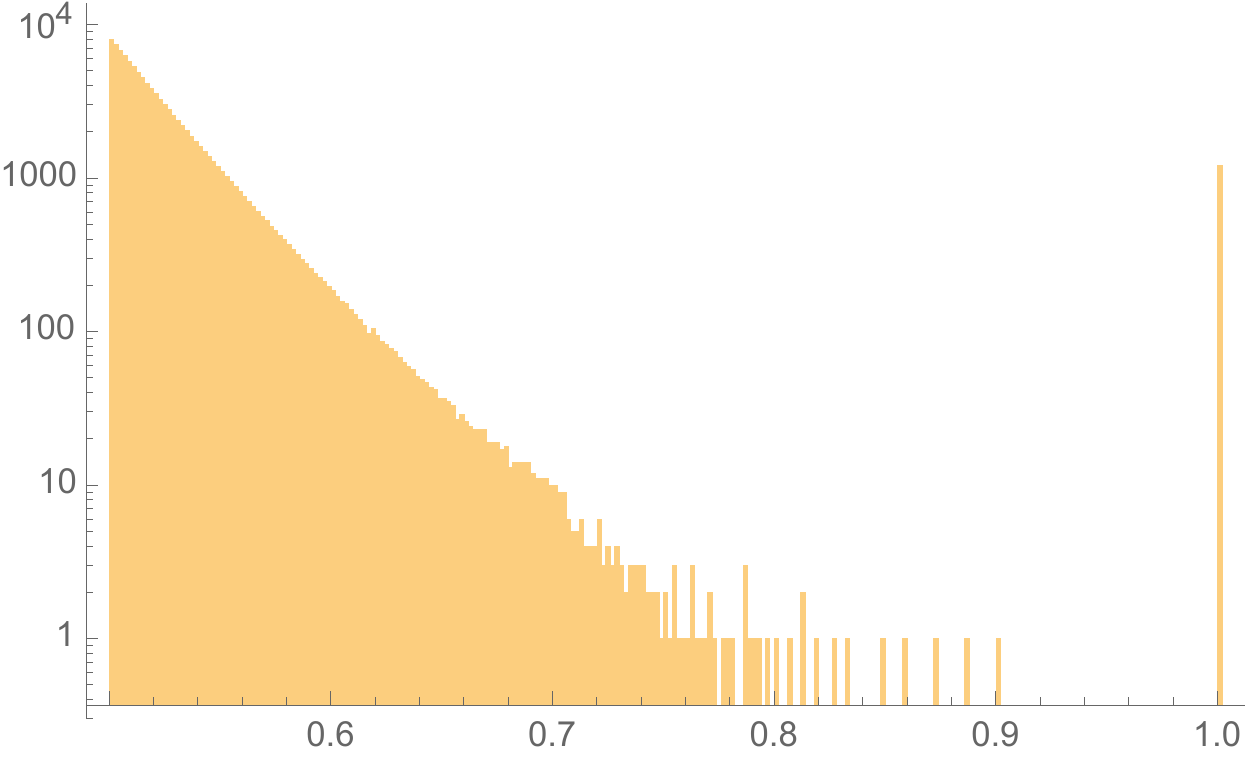}
\end{center}
\caption{Log-scale histograms of the distributions of the positive roots (left) and step roots (right) of the Littlewood polynomials of degree $\le 20$ and with zero-order coefficient $\rho_0=+1$. The algorithm found 2,255,683 roots and 106,682 step roots, where numbers such as $\alpha=1$ were counted  once each time they occurred as (step) roots of some polynomial.}
\label{roots fig}
\end{figure}

\section{Proofs of the results in Section~\ref{Takagi class section}}\label{Takagi class proofs}

\begin{proof}[Proof of Lemma~\ref{Rademacher f lemma}] Take $m\in\bN_0$ and let  $t\in[0,1]$ have Rademacher expansion $\bm\rho$. Then the tent map satisfies
$$\phi(t)=\frac14-\frac14\sum_{k=1}^\infty 2^{-k}\rho_0\rho_k\quad\text{and}\quad \phi(2^mt)=\frac14-\frac14\sum_{k=1}^\infty 2^{-k}\rho_m\rho_{m+k}.
$$
 Plugging formula this into \eqref{Takagi class eq} gives the result.
\end{proof}

 By 
\begin{equation*}
f_{n}(t):=\sum_{m=0}^n c_m\phi(2^mt),\qquad t\in[0,1],
\end{equation*}
we will denote the corresponding truncated function. 

Let
\begin{equation*}
\bD_n:=\{k2^{-n}\,|\,k=0,\dots,2^n\}\end{equation*}
be the dyadic partition of $[0,1]$ of generation $n$. 
For $t\in\bD_n$, we define its set of neighbors in $\bD_n$ by
\begin{equation*}
\cN_n(t)=\{s\in\bD_n\,|\,|t-s|=2^{-n}\}.
\end{equation*}
If $s\in\cN_n(t)$, we will say that $s$ and $t$ are neighboring points in $\bD_n$. We are now going to analyze the maxima of the truncated function $f_{n}$. Since this function is affine on all intervals of the form $[k2^{-(n+1)},(k+1)2^{-(n+1)}]$, it is clear that its maximum must be attained on $\bD_{n+1}$. In addition, $f_n$ can have flat parts (e.g., $n=0$ and $c_0=0$), so that the set of maximizers of $f_n$ may be an uncountable set. In the sequel, we are only interested in the set
\begin{equation*}
\cM_n=\bD_{n+1}\cap\argmax f_n
\end{equation*}
of maximizers located in $\bD_{n+1}$.

\begin{definition}\label{maximizing edge def}For $n\in\bN_0$, a pair $(x_n,y_n)$ is called a \emph{maximizing edge of generation $n$} if the following conditions are satisfied:
\begin{enumerate}
\item $x_n\in\cM_n$;
\item   $y_n$ is a maximizer of $f_n$ in $\cN_{n+1}(x_n)$.
\end{enumerate}\end{definition}

The following lemma characterizes the maximizing edges of generation $n$ as the maximizers of $f_n$ over neighboring pairs in $\bD_{n+1}$. It will be a key result for our proof of Theorem~\ref{general Takagi step cond thm}.

\begin{lemma}\label{key lemma}For $n\in\bN_0$, the following conditions are equivalent for two neighboring points $x_n,\,y_n\in\bD_{n+1}$.
\begin{enumerate}
\item $(x_n,y_n)$ or $(y_n,x_n)$ is a maximizing edge of generation $n$.
\item For all neighboring points $z_0,z_1$ in $\bD_{n+1}$, we have $f_n(z_0)+f_n(z_1)\le f_n(x_n)+f_n(y_n)$. 
\end{enumerate}
\end{lemma}

\begin{proof}We prove the assertion by induction on $n$. Consider the case $n=0$. If $c_0=0$, then $\cM_0=\bD_1$ and all pairs  of neighboring points in $\bD_1$ form maximizing edges of generation $0$, and so the assertion is obvious. 
If $c_0>0$, then $\cM_0=\{1/2\}$, and  if $c_0<0$, then $\cM_0=\{0,1\}$. Also in these cases the equivalence of (a) and (b) is obvious.

Now assume that $n\ge1$ and that the equivalence of (a) and (b)  has been established for all $m<n$. To show that (a) implies (b), let $(x_n,y_n)$ be a maximizing edge of generation $n$. First, we consider  the case $x_n\in\bD_{n+1}\setminus\bD_n$.  Then $\cN_{n+1}(x_n)$ contains $y_n$ and another point, say $u_n$, and both $y_n$ and $u_n$ belong to $\bD_n$. If $z_0$ and $z_2$ are two neighboring points in $\bD_n$, we let $z_1:=\frac12(z_0+z_2)$. Then
$$\frac12\big(f_{n-1}(z_0)+f_{n-1}(z_2)\big)+\frac{c_n}2=f_n(z_1)\le f_n(x_n)=\frac12\big(f_{n-1}(y_n)+f_{n-1}(u_n)\big)+\frac{c_n}2,
$$
and hence $f_{n-1}(z_0)+f_{n-1}(z_2)\le f_{n-1}(y_n)+f_{n-1}(u_n)$. The induction hypothesis now yields that $(y_n,u_n)$ or $(u_n,y_n)$ is a maximizing edge of generation $n-1$. Moreover, since $(x_n,y_n)$ is a maximizing edge of generation $n$, part (b) of Definition~\ref{maximizing edge def} gives $f_n(u_n)\le f_n(y_n)$. Since  both $y_n$ and $u_n$ belong to $\bD_n$, we get that 
\begin{equation}\label{key lemma eq 2}
f_{n-1}(u_n)=f_n(u_n)\le f_n(y_n)=f_{n-1}(y_n).
\end{equation}
Therefore, $y_n\in\cM_{n-1}$.

Now let  $z_0$ and $z_1$ be two neighboring points in $\bD_{n+1}$. Then one of the two, say $z_0$ belongs to $\bD_n$. Hence, the fact that $y_n\in\cM_{n-1}$ and $x_n\in\cM_n$ yields that 
 $$f_n(z_0)+f_n(z_1)=f_{n-1}(z_0)+f_n(z_1)\le f_{n-1}(y_n)+f_n(x_n)= f_{n}(y_n)+f_n(x_n).
 $$
 This establishes (b) in case $x_n\in\bD_{n+1}\setminus\bD_n$. 
 
 Now we consider the case in which $x_n\in\bD_n$. Then $f_{n-1}(x_n)=f_n(x_n)$, and so $x_n\in\cM_{n-1}$. Next, we let $y_{n-1}:=2y_n-x_n$. Then $y_{n-1}\in\bD_{n}$, and we claim that $(x_n,y_{n-1})$ is a maximizing edge of generation $n-1$. This is obvious if $x_n\in\{0,1\}$. Otherwise, we have $\cN_{n}(x_n)=\{y_{n-1},u_{n-1}\}$ for $u_{n-1}=2x_n-y_{n-1}=3x_n-2y_n$. Moreover, $\cN_{n+1}(x_n)=\{y_n,u_n\}$ for $u_n=\frac12(x_n+u_{n-1})$. Since $(x_n,y_n)$ is a maximizing edge of generation $n$, we must have $ f_n(y_n)\ge f_n(u_n)$ and hence
  \begin{align*}
\frac12\big(f_{n-1}(x_n)+f_{n-1}(y_{n-1})\big)+\frac{c_n}2=f_n(y_n)\ge f_n(u_n)=\frac12\big(f_{n-1}(x_n)+f_{n-1}(u_{n-1})\big)+\frac{c_n}2.
 \end{align*}
 Therefore, 
 \begin{equation}\label{key lemma eq 1}
 f_{n-1}( y_{n-1})\ge f_{n-1}(u_{n-1}),
 \end{equation}
  and it follows that $(x_n,y_{n-1})$ is indeed a maximizing edge of generation $n-1$.
 
Now let $z_0$ and $z_1$ be two neighboring points in $\bD_{n+1}$. Exactly one of these points, say $z_0$, belongs also to $\bD_n$. Let $z_2:=2z_1-z_0\in\bD_n$, so that $z_0$ and $z_2$ are neighboring points in $\bD_n$ and $z_1=\frac12(z_0+z_2)$. Hence, 
$f_n(z_1)=\frac12\big(f_{n-1}(z_0)+f_{n-1}(z_2)\big)+\frac{c_n}2$. Therefore, the fact that $ f_n(z_0)\le  f_n(x_n)$ and the induction hypothesis yield that
\begin{align*}
f_n(z_0)+f_n(z_1)&\le f_n(x_n)+ \frac12\big(f_{n-1}(x_{n})+f_{n-1}(y_{n-1})\big)+\frac{c_n}2= f_n(x_n)+ f_n(y_n).
\end{align*}
This completes the proof of (a)$\Rightarrow$(b).

Now we prove (b)$\Rightarrow$(a). To this end, let $x_n$ and $y_n$ be two fixed neighboring points in $\bD_{n+1}$ such that (b) is satisfied. Without loss of generality, we may suppose that $f_n(x_n)\ge f_n(y_n)$. Clearly,  $y_n$ must be a maximizer of $f_n$ in $\cN_{n+1}(x_n)$. To conclude (a), it will thus be sufficient to show that $x_n\in\cM_n$. To this end, we first consider the case $x_n\in\bD_n$. In a first step, we claim that $x_n\in\cM_{n-1}$. To this end, we assume by way of contradiction that there is $z_0\in\cM_{n-1}$ such that $f_{n-1}(z_0)>f_{n-1}(x_n)=f_n(x_n)$. Then we take $z_2\in\bD_n$ such that $(z_0,z_2)$ is a maximizing edge of generation $n-1$ and define $z_1:=\frac12(z_0+z_2)$ and $y_{n-1}:=2y_n-x_n\in\bD_n$. Using our assumption (b) yields that 
\begin{align*}
\lefteqn{f_{n-1}(x_n)+\frac12\big(f_{n-1}(x_n)+f_{n-1}(y_{n-1})\big)+\frac{c_n}2}\\
&=f_n(x_n)+f_n(y_n)\ge f_n(z_0)+f_n(z_1)\\
&=f_{n-1}(z_0)+\frac12\big(f_{n-1}(z_0)+f_{n-1}(z_2)\big)+\frac{c_n}2\\
&>f_{n-1}(x_n)+\frac12\big(f_{n-1}(z_0)+f_{n-1}(z_2)\big)+\frac{c_n}2.
\end{align*}
Hence, $f_{n-1}(x_n)+f_{n-1}(y_{n-1})>f_{n-1}(z_0)+f_{n-1}(z_2)$, in contradiction to  our assumption that $(z_0,z_2)$ is a maximizing edge of generation $n-1$ and the induction hypothesis. Therefore we must have $x_n\in\cM_{n-1}$.

In the next step, we show that $f_n(x_n)\ge f_n(z)$ for all $z\in\bD_{n+1}\setminus\bD_n$. Together with the preceding step, this will  give $x_n\in\cM_n$. To this end, let $z_1\in\bD_{n+1}\setminus\bD_n$ be given, and let $z_0$ and $z_1$ be the two neighbors of $z_1$ in $\bD_{n+1}$. Then $z_0,z_2\in\bD_n$ and $z_1=\frac12(z_0+z_2)$. As discussed above, $y_n$ is a maximizer of $f_n$ in $\cN_{n+1}(x_n)$. Thus, it is easy to see that $y_{n-1}:=2y_n-x_n$ must be a maximizer of $f_{n-1}$ in $\cN_{n}(x_n)$. Since we already know that $x_n\in\cM_{n-1}$, the induction hypothesis yields  that $(x_n,y_{n-1})$ is a maximizing edge of generation $n-1$. Thus,
\begin{align*}
f_n(z_1)=\frac12\big(f_{n-1}(z_0)+f_{n-1}(z_2)\big)+\frac{c_n}2\le \frac12\big(f_{n-1}(x_n)+f_{n-1}(y_{n-1})\big)+\frac{c_n}2=f_n(y_n)\le f_n(x_n).
\end{align*}
This concludes the proof of (b)$\Rightarrow$(a) in case $x_n\in\bD_n$. 

Now we consider the case in which $x_n\in\bD_{n+1}\setminus\bD_n$. In a first step, we show that $y_n\in\cM_{n-1}$. To this end, we assume by way of contradiction that there is $z_0\in\cM_{n-1}$ such that $f_{n-1}(z_0)>f_{n-1}(y_n)$. Let $z_2\in\bD_n$ be such that $(z_0,z_2)$ is a maximizing edge of generation $n-1$ and put $z_1:=\frac12(z_0+z_2)$. We also put $u_n:=2x_n-y_n$. Then the induction hypothesis gives
\begin{align*}
f_n(z_0)+f_n(z_1)&=f_{n-1}(z_0)+\frac12\big(f_{n-1}(z_0)+f_{n-1}(z_2)\big)+\frac{c_n}2\\
&>f_{n-1}(y_n)+\frac12\big(f_{n-1}(y_n)+f_{n-1}(u_n)\big)+\frac{c_n}2\\
&=f_n(y_n)+f_n(x_n),
\end{align*}
in contradiction to our assumption (b). Thus, $y_n\in\cM_{n-1}$. 

Next, we show that $(y_n,u_n)$ is a maximizing edge of generation $n-1$. This is clear if either $y_n$ or $u_n$ belong to $\{0,1\}$. Otherwise, we must show that $f_{n-1}(u_n)\ge f_{n-1}(w_n)$, where $w_n=2y_n-u_n$. 
Let $z_n:=\frac12(w_n+y_n)$. Then our hypothesis (b) yields that $f_n(y_n)+f_n(x_n)\ge f_n(y_n)+f_n(z_n)$ and in turn $f_n(x_n)\ge f_n(z_n)$. It follows that
\begin{align*}
\frac12\big(f_{n-1}(y_n)+f_{n-1}(u_n)\big)+\frac{c_n}2=f_n(x_n)\ge f_n(z_n)=\frac12\big(f_{n-1}(y_n)+f_{n-1}(w_n)\big)+\frac{c_n}2,
\end{align*}
which implies the desired inequality $f_{n-1}(u_n)\ge f_{n-1}(w_n)$.

Now we can conclude our proof by showing that $x_n\in\cM_n$. If $z\in\bD_n$, then the fact that $y_n\in\cM_{n-1}$ gives 
$$f_n(x_n)\ge f_n(y_n)=f_{n-1}(y_n)\ge f_{n-1}(z)=f_n(z).$$
If $z\in\bD_{n+1}\setminus\bD_n$, we let $z_0$ and $z_2$ denote its two neighboring points in $\bD_{n+1}$, so that $z_0,z_2\in\bD_n$ and $z=\frac12(z_0+z_2)$. 
Then,
$$f_n(z)=\frac12\big(f_{n-1}(z_0)+f_{n-1}(z_2)\big)+\frac{c_n}2\le \frac12\big(f_{n-1}(y_n)+f_{n-1}(u_n)\big)+\frac{c_n}2=f_n(x_n),
$$
where we have used the induction hypothesis and the fact that $(y_n,u_n)$ is a maximizing edge of generation $n-1$.
\end{proof}

	 In the proof of the preceding lemma (see, in particular, \eqref{key lemma eq 2} and \eqref{key lemma eq 1}), we have en passant proved the following statement, which shows how to successively construct maximizing edges in a backward manner. 
	 
	 \begin{lemma}Suppose that  $(x_n,y_n)$ is a maximizing edge of generation $n\ge1$. Then:
	 \begin{enumerate}
	 \item If $x_n\in\bD_n$ and $y_{n-1}:=2y_n-x_n$, then $(x_n,y_{n-1})$ is a maximizing edge of generation $n-1$.
	 \item If $x_n\in\bD_{n+1}\setminus\bD_n$ and $u_n:=2x_n-y_n$, then  $(y_n,u_n)$ is a maximizing edge of generation $n-1$.
	 \end{enumerate}
	 \end{lemma}

	 We also have the following result, which shows how maximizing edges can be constructed in a forward manner. 
	 
	 \begin{lemma}\label{max edge lemme going up}For $n\in\bN_0$ let $(x_n,y_n)$ be a maximizing edge of generation $n$ and define $z_n:=\frac12(x_n+y_n)$. Then $(x_n,z_n)$ or $(z_n,x_n)$ is a maximizing edge of generation $n+1$.
	 \end{lemma}

	 \begin{proof}Note that $f_{n+1}(z_n)=\frac12(f_n(x_n)+f_n(y_n))+c_{n+1}/2$. Hence, property (b) in Lemma~\ref{key lemma} yields  that $f_{n+1}(z_n)\ge f_{n+1}(z)$ for all $z\in\bD_{n+2}\setminus\bD_{n+1}$. Moreover, by assumption, $f_{n+1}(x_n)	=f_n(x_n)\ge f_n(z)=f_{n+1}(z)$ for all $z\in\bD_{n+1}$. Hence, $z_n\in\cM_{n+1}$ if $f_{n+1}(z_n) \ge f_{n+1}(x_n)$ and $x_n\in\cM_{n+1}$ if $f_{n+1}(x_n) \ge f_{n+1}(z_n)$. From here, the assertion follows easily.\end{proof}

The next proposition states in particular, that $t$ is a maximizer of $f$ if and only if it is a limit of successive maximizers of $f_n$. Clearly, the \lq\lq if" direction of this statement is obvious, while the \lq\lq only if" direction is not.

\begin{proposition}\label{max convergence prop}For given $t\in[0,1]$, the following statements are equivalent.
\begin{enumerate}
\item $t\in\argmax f$.
\item There exists a sequence $(t_n)_{n\in\bN_0}$ such that $t_n\in\cM_n$ for all $n$ and $\lim_nt_n=t$.
\item For $n\in\bN_0$, let $E_n$ be the union of all intervals $[x,y]$ such that $x,t\in\bD_{n+1}$,  $x<y$ and $(x,y)$ or $(y,x)$ is a maximizing edge of generation $n$. Then 
$$t\in\bigcap_{n=0}^\infty E_n.
$$
\end{enumerate}
\end{proposition}

\begin{proof}To prove (a)$\Rightarrow$(c), we assume by way of contradiction that there is $n\in\bN_0$ such that $t\notin E_n$. Clearly,  we can take the smallest such $n$. Since   $E_0=[0,1]$, we must have $n\ge1$. Moreover, there must be   a maximizing edge of generation $n-1$, denoted $(x_{n-1},y_{n-1})$, such that $t$ belongs to the closed interval with endpoints $x_{n-1}$ and $y_{n-1}$. Let $z:=\frac12(x_{n-1}+y_{n-1})$. By Lemma~\ref{max edge lemme going up}, the closed interval with endpoints $x_{n-1}$ and $z$ is a subset of $E_n$. Hence, $t\neq z$ and $t$ must be contained in the half-open interval with endpoints $z$ and $y_{n-1}$.  Therefore, $t=\alpha y_{n-1}+(1-\alpha)z$ for some $\alpha\in(0,1]$. We define $s:=\alpha x_{n-1}+(1-\alpha)z=2z -t$.

Since the interval with endpoints $z$ and $y_{n-1}$  is not a subset of $E_n$, Lemma~\ref{key lemma} implies that $f_n(z)+f_n(y_{n-1})<f_n(z)+f_n(x_{n-1})$. As $f_n$ is affine on each the two respective intervals with endpoints $y_{n-1}, z$ and $z,x_{n-1}$, we thus get $f_n(s)>f_n(t)$. Moreover, the symmetry and periodicity of the tent map $\phi$ implies that $\phi(2^mt)=\phi(2^ms)$ for all $m> n$. Hence,
$$f(s)=f_n(s)+\sum_{m=n+1}^\infty c_m\phi(2^ms)>f_n( t)+\sum_{m=n+1}^\infty c_m\phi(2^m t)=f(t),
$$
which  contradicts  the assumed maximality of $t$. 

The implication (c)$\Rightarrow$(b) is obvious, because $|x_n-y_n|=2^{-(n+1)}$ whenever $(x_n,y_n)$ is a  maximizing edge of generation $n$. The implication (b)$\Rightarrow$(a) follows from the uniform convergence of $f_n$ to $f$. 
 \end{proof}

The following lemma expresses the slope of $f_n$ around a point $t\in[0,1]$ in terms of the Rademacher expansion of $t$. 

\begin{lemma}\label{slope lemma}For a given $\{-1,+1\}$-valued sequence $(\rho_m)_{m\in\bN_0}$ and $n\in\bN$ let
$$t_n:=\sum_{m=0}^n(1-\rho_m)2^{-(m+2)}.
$$
Then
$$\frac{f_n(y)-f_n(x)}{y-x}=\sum_{m=0}^n2^mc_m\rho_m\qquad\text{for all $x,y\in[t_n,t_n+2^{-(n+1)}]$ with $x\neq y$.}
$$
\end{lemma}

\begin{proof}We proceed by induction on $n$. For $n=0$, we have $f_0=c_0\phi$ and  $\rho_0=-1$ if and only if $t_0=1/2$; otherwise we have  $t_0=0$. Hence, the assertion is obvious.

Now assume that $n\ge1$, that the assertion has been established for all $m<n$, and that $x,y\in[t_n,t_n+2^{-(n+1)}]$ are given. Then $x$ and $y$ also belong to $[t_{n-1},t_{n-1}+2^{-n}]$, and so the induction hypothesis yields that
\begin{align}
\frac{f_n(y)-f_n(x)}{y-x}&=\frac{f_{n-1}(y)-f_{n-1}(x)}{y-x} +c_n\frac{\phi(2^ny)-\phi(2^nx)}{y-x} \nonumber\\
&
=\sum_{m=0}^{n-1}2^mc_m\rho_m+c_n\frac{\phi(2^ny)-\phi(2^nx)}{y-x}.\label{phi recursion eq}
\end{align} 
To deal with the rightmost term, we write $x=t_{n-1}+\xi2^{-n}$ and $y=t_{n-1}+\eta 2^{-n}$, where $\xi,\eta\in[0,1]$. More precisely, $\xi,\eta\in[0,1/2)$ if $\rho_n=1$ and $\xi,\eta\in[1/2,1]$ if $\rho_n=-1$. Then the rightmost term in \eqref{phi recursion eq} can be expressed as follows,
\begin{align*}
c_n\frac{\phi(2^ny)-\phi(2^nx)}{y-x}&=c_n\frac{\phi(2^nt_{n-1}+\eta)-\phi(2^nt_{n-1}+\xi)}{y-x}=2^nc_n\frac{\phi(\eta)-\phi(\xi)}{\eta-\xi},
\end{align*}
where we have used the periodicity of $\phi$ and the fact that $2^nt_{n-1}\in\bZ$. By our choice of $\xi$ and $\eta$, the rightmost term is equal to $2^nc_n\rho_n$, which in view of \eqref{phi recursion eq} concludes the proof.\end{proof}

We need one additional lemma for the proof of Theorem~\ref{general Takagi step cond thm}.

\begin{lemma}\label{step lemma}Suppose that $(\rho_m)_{m\in\bN_0}$ is a $\{-1,+1\}$-valued sequence  and $n\in\bN_0$. If \begin{equation}\label{finite n step cond eq}
\rho_{k}\sum_{m=0}^{k-1}2^mc_m\rho_m\le0\qquad\text{for all $k\le n$,}
\end{equation}
then there exists a maximizing edge $(x_n,y_n)$ of generation $n$ such that $t:=\sum_{m=0}^\infty(1-\rho_m)2^{-(m+2)}$ belongs to the closed interval with endpoints $x_n$ and $y_n$.\end{lemma}

\begin{proof}We will prove the assertion by induction on $n$. If $n=0$, the hypothesis is trivially satisfied since both intervals $[0,1/2]$ and $[1/2,1]$ have endpoints that form maximizing edges of generation~0. 

Now suppose that $n\ge1$ and that the assertion has been established for all $m<n$. Let $(x_{n-1},y_{n-1})$ be the maximizing edge  of generation $n-1$ that contains $t$. 
Lemma~\ref{slope lemma} gives that
\begin{equation}\label{Deltan-1eq}
\Delta_{n-1}:=\frac{f_{n-1}(y_{n-1})-f_{n-1}(x_{n-1})}{y_{n-1}-x_{n-1}}=\sum_{m=0}^{n-1}2^mc_m\rho_m .
\end{equation}
Let 
$z:=\frac12(x_{n-1}+y_{n-1})$. If $\Delta_{n-1}=0$ or $c_n=0$, then Lemmas~\ref{key lemma} and~\ref{max edge lemme going up} imply that  $x_{n-1},z$ and $y_{n-1},z$ are the endpoints of two respective maximizing edges of generation $n$, of which at least one must enclose $t$. 
If $\Delta_{n-1}>0$, then we must have $x_{n-1}>y_{n-1}$, because the numerator in \eqref{Deltan-1eq} is strictly negative. Moreover, \eqref{finite n step cond eq}
 implies that  $\rho_n=-1$, which means that $t$ lies in the interval $[z,x_{n-1}]$, whose endpoints form a maximizing edge of generation $n$ according to Lemma~\ref{max edge lemme going up}. An analogous reasoning gives $t\in[x_{n-1},z]$ if $\Delta_{n-1}<0$.
\end{proof}

\begin{proof}[Proof of Theorem~\ref{general Takagi step cond thm}] (a)$\Rightarrow$(b): Suppose that there exists a Rademacher expansion $(\rho_m)_{m\in\bN_0}$ of $t$ that does not satisfy the step condition. Then there exists  $n\in\bN_0$ such that $\rho_{n+1}\sum_{m=0}^{n}2^mc_m\rho_m>0$. Let us fix the smallest such $n$. Then \eqref{finite n step cond eq} holds, and  Lemma~\ref{step lemma}  yields a maximizing edge of generation $n$, denoted $(x_n,y_n)$, such that $t$ belongs to the closed interval with endpoints $x_n,y_n$. Suppose first that $\Delta_n:=\sum_{m=0}^{n}2^mc_m\rho_m>0$. Lemma~\ref{slope lemma} gives that
\begin{equation}\label{general Takagi step cond thm eq1}
0\ge f_n(y_n)-f_n(x_n)=\Delta_n\cdot(y_n-x_n)
\end{equation}
 and hence that $y_n<x_n$. Moreover, we must have strict inequality in \eqref{general Takagi step cond thm eq1}. 
 
 Let $z=\frac12(x_n+y_n)$ so that $y_n<z<x_n$. Lemma~\ref{max edge lemme going up}   yields that either $(z,x_n)$ or $(x_n,z)$ is a maximizing edge of generation $n+1$.  Therefore, and since   $f_n(y_n)<f_n(x_n)$, Lemma~\ref{key lemma} implies that neither $(y_n,z)$ nor $(z,y_n)$ is a maximizing edge of generation $n+1$. 
But the fact that $\rho_{n+1}=1$ requires that $t$ belongs to $[y_{n},z]$. Therefore, Proposition~\ref{max convergence prop} yields that $t\notin\argmax f$. An analogous argument applies in case $\Delta_n<0$. 

(b)$\Rightarrow$(c) is obvious, and (c)$\Rightarrow$(a) follows from Lemma~\ref{step lemma} and Proposition~\ref{max convergence prop}.\end{proof}

The proof of Corollary~\ref{smallest largest cor} will be based on the following simple lemma. We denote by $\cR_+$ the set of all $\{-1,+1\}$-valued sequences $\bm\rho$ that satisfy the step condition and  $\rho_0=+1$.

\begin{lemma}\label{sign lemma}Suppose that $\bm\rho^{(1)}$ and $\bm\rho^{(2)}$ are two distinct sequences in $\cR_+$. If $n_0$ denotes the smallest $n\in\bN$ such that $\rho^{(1)}_n\neq\rho_n^{(2)}$, then $\sum_{m=0}^{n_0-1}2^mc_m\rho^{(i)}_m=0$ for $i=1,2$. 
\end{lemma}

\begin{proof} On the one hand, $\rho^{(1)}_m=\rho^{(2)}_m$ for $m<n_0$ and so 
$$\rho^{(1)}_{n_0}\sum_{m=0}^{n_0-1}2^mc_m\rho^{(1)}_m\le0\quad\text{and}\quad \rho^{(2)}_{n_0}\sum_{m=0}^{n_0-1}2^mc_m\rho^{(1)}_m\le0.
$$
On the other hand, $\rho^{(1)}_{n_0}=-\rho^{(2)}_{n_0}$. This proves the assertion. 
\end{proof}

\begin{proof}[Proof of Corollary~\ref{smallest largest cor}]Since both $\bm\rho^\flat$ and $\bm\rho^\sharp$ satisfy the step condition and since $\rho^\flat_0=\rho^\sharp_0=1$,  both $t^\flat:=\cT(\bm\rho^\flat)$  and $t^\sharp:=\cT(\bm\rho^\sharp)$ belong to $[0,1/2]\cap\argmax f$. Now suppose that there exists $t\in[0,1/2]\cap\argmax f$ with $t\neq t^\flat$. Let $\bm\rho$ be the standard Rademacher expansion for $t$ and take $n_0$ as in Lemma~\ref{sign lemma}. Then this lemma gives that  the first $n_0-1$ coefficients in the binary expansions of $t^\flat$ and $t$ coincide.  Moreover, the definition of $\bm\rho^\flat$ in \eqref{rho+- eq}
 yields that $\rho^\flat_{n_0}=-1$ and hence that $\rho_{n_0}=+1$. Therefore,  the $n_0^{\text{th}}$ coefficients in the binary expansions of $t^\flat$ and $t$ are given by  $1$  and $0$, respectively. Since $\bm\rho$ is the standard Rademacher expansion for $t$, the corresponding binary expansion of $t$ must contain infinitely many zeros, and so $t^\flat$ must be strictly larger than $ t$. The proof for $t^\sharp$ is analogous.\end{proof}

 \begin{proof}[Proof of Proposition~\ref{Takagi class number of max prop}]
 First, we will consider the case $|\cZ|<\infty$
  and proceed by induction on $n:=|\cZ|$. If $n=0$, then Lemma~\ref{sign lemma}  implies that $\bm\rho^\sharp$ is the only sequence in $\cR_+$. 
 Now suppose that $n\ge1$ and that the assertion has been established for all $m< n$. We let $n_0:=\min\cZ$.  If $\bm\rho$ is any sequence in $\cR_+$, then $\rho_k=\rho_k^\sharp$ for all $k\le n_0$. Hence, for any $n> n_0$,
 \begin{align}\label{cR card lemma eq1}
 \sum_{m=0}^n2^mc_m\rho_m=\sum_{m=0}^{n_0}2^mc_m\rho^\sharp_m+\sum_{m=n_0+1}^n2^mc_m\rho_m=\sum_{m=0}^{n-n_0-1}2^m\wt c_m\wt\rho_m,
 \end{align}
 where $\wt c_m=2^{n_0}c_{m+n_0+1}$ and $\wt\rho_m=\rho_{m+n_0+1}$. It follows in particular that $\bm{\wt\rho}$ satisfies the step condition for $(\wt c_m)_{m\in\bN_0}$. 
 
 Next, we define $\wt\rho^\sharp_m:=\rho^\sharp_{m+n_0+1}$ and observe that $\wt\rho^\sharp_0=+1$. Moreover, \eqref{cR card lemma eq1} implies that $\bm{\wt\rho}^\sharp$  is indeed the $\sharp$-sequence for $(\wt c_m)_{m\in\bN_0}$. Let
 $$\wt \cZ:=\Big\{n\in\bN_0\,\Big|\,\sum_{m=0}^n2^m\wt c_m\wt \rho_m^\sharp=0\Big\}
$$
and denote by  $\wt\cR_+$ the class of all $\{-1,+1\}$-valued sequences $\wt{\bm\rho}$ with $\wt\rho_0=+1$. Then $|\wt\cZ|=n-1$, and the induction hypothesis implies that $|\wt\cR_+|=2^{n-1}$. 
 The set $\wt\cR_+$ corresponds to all sequences  $\bm\rho\in\cR_+$ that satisfy $\rho_{n_0}=+1$. Now let us introduce the set $\wt \cR_-$ of all sequences $\bm{\wt\rho}$  with $\wt\rho_0=-1$ that satisfy the step condition for  $(\wt c_m)_{m\in\bN_0}$. Then $|\cR_+|=|\wt\cR_+|+|\wt\cR_-|$. But it is clear that we must have $|\wt\cR_-|=|\wt\cR_+|$, because if $\bm\rho$ satisfies the step condition, then so does $-\bm\rho$. This concludes the proof if $|\cZ|<\infty$.
 
Now consider the case $|\cZ|=\infty$. We write $\cZ\cup\{0\}=\{n_0,n_1,\dots\}$. For every sequence $\bm\sigma\in\{-1,+1\}^{\bN_0}$ with $\sigma_0=+1$, we define a sequence $\bm\rho^{\bm\sigma}$ by
$\rho^{\bm\sigma}_m:=\sigma_i\rho_m^\flat$ if $n_i< m\le n_{i+1}$. One easily checks that $\bm\rho\in\cR_+$ and it is clear that $\bm\rho^{\bm\sigma}\neq\bm\rho^{\bm\eta}$ if  $\bm\eta$ is another sequence in $\{-1,+1\}^{\bN_0}$ with $\eta_0=+1$. Therefore, $\cZ$ has the cardinality of the continuum.\end{proof}

\section{Proofs for results in Section~\ref{TL section}}\label{TL proofs}

\begin{proof}[Proof of Proposition~\ref{continuity prop}] The equivalence of~\ref{continuity prop (a)}  and~\ref{continuity prop (b)} is obvious. In addition, it is easy to see that~\ref{continuity prop (b)} is equivalent to $\bm\rho^\flat(\alpha)=\bm\rho^\sharp(\alpha)$, which in turn is equivalent to~\ref{continuity prop (c)} by \eqref{rho+- eq bis}. 

Let us now show that~\ref{continuity prop (c)} implies~\ref{continuity prop (d)}. To this end, we first show by induction on $n$ that for every $n\in\bN_0$ there exists $\delta_n>0$ such that $\rho^\flat_k(\beta)=\rho^\sharp_k(\alpha)$ for $k\le n$ and $\beta\in(-2,2)$ with $|\alpha-\beta|<\delta_n$. This is obvious for $n=0$. If the assertion has been established for $n$, then 
\begin{equation*}
L_n^\sharp(\beta):=\sum_{m=0}^n\beta^m\rho^\sharp_m(\beta)=\sum_{m=0}^n\beta^m\rho^\sharp_m(\alpha)=:g_n(\beta)
\end{equation*}
for $|\beta-\alpha|<\delta_n$. Since $g_n$ is clearly continuous and $g_n(\alpha)\neq0$ by~\ref{continuity prop (c)}, there exists $\delta_{n+1}\in(0,\delta_n]$ such that $0<g_n(\beta)g_n(\alpha)$ for all $\beta$ with $|\beta-\alpha|<\delta_{n+1}$. But then we must also have $L_n^\sharp(\beta)L_n^\sharp(\alpha)>0$ for these $\beta$, and \eqref{rho+- eq bis} implies that $\rho_{n+1}^\sharp(\beta)=\rho_{n+1}^\sharp(\alpha)$.

To show the continuity of $\tau^\sharp$, we let $\eps>0$ be given and define $n:=\lfloor -\log_2\eps\rfloor$. Then the preceding step yields that for $|\beta-\alpha|<\delta_n$,
\begin{align*}
|\tau^\sharp(\beta)-\tau^\sharp(\alpha)|&=\bigg|\sum_{m=0}^\infty2^{-(m+2)}\big(\rho^\sharp_m(\alpha)-\rho^\sharp_m(\beta)\big)\bigg|\le \sum_{m=n+1}^\infty 2^{-(m+1)}=2^{-n-1}<\eps.
\end{align*}
The continuity of $\tau^\flat$ is proved in the same way.

Finally, we show that~\ref{continuity prop (d)} implies~\ref{continuity prop (a)}. To this end, let us assume that, e.g., $\tau^\sharp$ is not continuous at $\alpha$. Then there are two sequences $(\alpha_n)$ and $(\beta_n)$ in $(-2,2)$ such that $\alpha_n\to\alpha$ and $\beta_n\to\alpha$, but $t_0:=\lim_n\tau^\sharp(\alpha_n)\neq\lim_n\tau^\sharp(\beta_n)=:t_1$. Since $f_{\beta}(t)\to f_{\alpha}(t)$ uniformly in $t$ as $\beta\to\alpha$, it follows that $t_0$ and $t_1$ are both maximizers of $f_\alpha$. Hence,~\ref{continuity prop (a)} cannot hold.
\end{proof}

The following lemma uses a result from Moran~\cite{Moran} so as to determine the Hausdorff dimension of certain sets in $[0,1]$ that are defined in terms of the Rademacher expansions of their members. Of course, instead of the Rademacher expansion, we could have just as well used the binary expansion.

\begin{lemma}\label{Hausdorff lemma}For a given integer $n\ge2$ and $k=0,\dots, n-1$, let $\rho_k\in\{-1,1\}$ and $\rho_k^*=-\rho_k$. Let $C$ be the set of all numbers in $[0,1]$ that have a Rademacher expansion composed of successive blocks of the form $\rho_0,\rho_1,\cdots,\rho_{n-1}$ or $\rho^*_0,\rho^*_1,\cdots,\rho^*_{n-1}$. Then $C$ is a perfect set of Hausdorff dimension $1/n$ and  the $1/n$-dimensional Hausdorff measure of $C$ is finite and strictly positive.
\end{lemma}

\begin{proof}It is clear that $C$ is closed and that every point  $t\in C$ is the limit of some sequence in $C\setminus\{t\}$. Therefore, $C$ is perfect.

Next, $C$ is the disjoint union of the two sets $C_1$ and $C_1^*$ that consist of all numbers $t\in C$ that have a Rademacher expansion whose first $n$ digits are formed by the blocks $\rho_0,\rho_1,\cdots,\rho_{n-1}$ and $\rho^*_0,\rho^*_1,\cdots,\rho^*_{n-1}$, respectively. Clearly, the two sets $C_1$ and $C_1^*$ are similar geometrically to $C$ but reduced in size by a factor $2^{-n}$. It therefore follows from~\cite[Theorem II]{Moran} that $C$ has Hausdorff dimension $\log 2/\log 2^n=1/n$ and that the $1/n$-dimensional Hausdorff measure of $C$ if finite and strictly positive. \end{proof}

\subsection{Proofs for the results in Section~\ref{alpha<1 results section}}

\begin{proof}[Proof of Lemma~\ref{Littlewood neg root lemma}] Note that
\begin{equation*}
p_{2n}(x)=1-\frac{x(1-x^{2n})}{1-x}=\frac{q_{2n}(x)}{1-x}
\end{equation*}
for $q_{2n}(x)=1-2x+x^{2n+1}$. On the one hand, if $x\le-2$, then $q_{2n}(x)=1+x(x^{2n}-2)\le-3$. On the other hand, for $x\in[-1,0)$, we have $q_{2n}(x)\ge-2x>0$. Therefore, all negative roots of $q_{2n}$, and equivalently of $p_{2n}$, must be contained in $(-2,-1)$. Next, 
\begin{equation*}
q'_{2n}(x)=-2+(2n+1)x^{2n}>0\qquad\text{for $x\in(-2,-1)$,}
\end{equation*}
 which together with $q_{2n}(-2)<0$ and $q_{2n}(-1)=2$ yields the existence of a unique negative root, which belongs to $(-2,-1)$. This observation  furthermore yields that for $x\in(-2,-1)$,
\begin{equation}\label{Littlewood neg root lemma p2n sign eq}
\text{$p_{2n}(x)<0$ for $x<x_n$ and $p_{2n}(x)>0$ for $x>x_n$.}
\end{equation} 
From here, we also get  $x_{n+1}>x_{n}$, because
$$q_{2n+2}(x_{n})=q_{2n}(x_{n})+x_{n}^{2n+3}-x_{n}^{2n+1}=x_{n}^{2n+3}-x_{n}^{2n+1}<0.
$$

Finally, we show that $\lim_nx_{n}=1$. To this end, we assume by way of contradiction that $x_\infty:=\lim_nx_{n}$ is strictly less than 1. Since $x_\infty>x_{n}$ for all $n$, \eqref{Littlewood neg root lemma p2n sign eq} gives
$$0\le \lim_{n\ua\infty}q_{2n}(x_\infty)=1-2x_\infty+\lim_{n\ua\infty}x_{\infty}^{2n+1}=-\infty,
$$
which is the desired contradiction.
\end{proof}

For the ease of notation, we define
\begin{equation*}
R^\flat_n(\alpha):=\sum_{m=0}^{n}\alpha^m\rho^\flat_m(\alpha)\qquad\text{and}\qquad R^\sharp_n(\alpha):=\sum_{m=0}^{n}\alpha^m\rho^\flat_m(\alpha).
\end{equation*}

\begin{lemma}\label{alpha negative first rho entries lemma}In the setting of Theorem~\ref{alpha negative thm}, we have for $\alpha\in(-2,-1)$ and $n\in\bN_0$,
\begin{align*}
\rho^\flat_1(\alpha)&=\cdots=\rho^\flat_{2n+1}(\alpha)=-1\qquad\text{for $\alpha\in[x_n,x_{n+1})$,}\\
\rho^\sharp_1(\alpha)&=\cdots=\rho^\sharp_{2n+1}(\alpha)=-1\qquad\text{for $\alpha\in(x_n,x_{n+1}]$.}
\end{align*}
Moreover, for $m< 2n$ we have $R_m^\flat(\alpha)> 0$, and we have $R_{2n}^\flat(\alpha)\ge0$, where equality holds if and only if $\alpha=x_n$.
\end{lemma}

\begin{proof}We prove only the result for $\bm\rho^\flat$; the proof for $\bm\rho^\sharp$ is analogous. To this end, we note first that $R^\flat_0(\alpha)=1$ so that $\rho^\flat_1(\alpha)=-1$. This settles the case $n=0$. For arbitrary $n\in\bN$,  we now show by induction on $m\in\{1,\dots, n\}$ that $R^\flat_{2m-1}(\alpha)>0$ and $R^\flat_{2m}(\alpha)\ge0$ with equality if and only if $m=n$ and $\alpha=x_n$. Consider the case $m=1$. We have 
$R_1^\flat(\alpha)=1-\alpha>0$ and, hence, $R_2^\flat(\alpha)=p_2(\alpha)$, where $p_{2m}$ denotes the Littlewood polynomial introduced in Lemma~\ref{Littlewood neg root lemma}. Since $\alpha\ge x_{n}$ by assumption and $x_n\ge x_{1}$ by Lemma~\ref{Littlewood neg root lemma},  the observation \eqref{Littlewood neg root lemma p2n sign eq} gives $p_{2}(\alpha)\ge0$, with equality if and only if $n=1$ and $\alpha=x_1$.

If the assertion has been proved for all $k< m\le n$, then the induction hypothesis implies that $R^\flat_{2m-1}(\alpha)=R^\flat_{2m-2}(\alpha)-\alpha^{2m-1}\ge -\alpha^{2m-1} >0$. The induction hypothesis  implies moreover that $R^\flat_{2m}(\alpha)=p_{2m}(\alpha)$.  Since $\alpha\ge x_n$ by assumption and $x_n\ge x_{m}$ by Lemma~\ref{Littlewood neg root lemma},  the observation \eqref{Littlewood neg root lemma p2n sign eq} gives $p_{2m}(\alpha)\ge0$, with equality if and only if $m=n$ and $\alpha=x_n$. This completes the  proof. 
\end{proof}

\begin{lemma}\label{alpha negative sequence lemma} In the setting of Theorem~\ref{alpha negative thm}, we have  for  $m,n\in\bN_0$,
\begin{align*}
R^\flat_{2n+4m+1}>0,\ R^\flat_{2n+4m+2}<0,\ R^\flat_{2n+4m+3}<0,\ R^\flat_{2n+4m+4}>0\qquad\text{on $[x_n,x_{n+1})$,}\\
R^\sharp_{2n+4m+1}>0,\ R^\sharp_{2n+4m+2}<0,\ R^\sharp_{2n+4m+3}<0,\ R^\sharp_{2n+4m+4}>0\qquad\text{on $(x_n,x_{n+1}]$.}
\end{align*}
\end{lemma}

\begin{proof}We prove only the result for $R^\flat$; the proof for $R^\sharp$ is analogous.
We fix $n\in\bN_0$ and $\alpha\in [x_n,x_{n+1})$ (and $\alpha>x_0=-2$ for $n=0$) and proceed by induction on $m$. For $m=0$ we get from Lemma~\ref{alpha negative first rho entries lemma} and \eqref{Littlewood neg root lemma p2n sign eq} that $R^\flat_{2n+1}(\alpha)=p_{2n}(\alpha)-\alpha^{2n+1}>p_{2n}(\alpha)\ge0 $. Therefore $\rho^\flat_{2n+2}(\alpha)=-1$ and so  $R^\flat_{2n+2}(\alpha)=p_{2n+2}(\alpha)<0$ by Lemma~\ref{Littlewood neg root lemma}. In turn, we get $\rho^\flat_{2n+3}(\alpha)=+1$ and so $R^\flat_{2n+3}(\alpha)=R^\flat_{2n+2}(\alpha)+\alpha^{2n+3}<0$. Therefore, $\rho^\flat_{2n+4}(\alpha)=+1$ and, finally, 
\begin{align*}
R^\flat_{2n+4}(\alpha)&=R^\flat_{2n}(\alpha)-\alpha^{2n+1}-\alpha^{2n+2}+\alpha^{2n+3}+\alpha^{2n+4}\\
&=R^\flat_{2n}(\alpha)-\alpha^{2n+1}(1+\alpha-\alpha^2-\alpha^3)>0,
\end{align*}
where we have used that $R^\flat_{2n}(\alpha)\ge0$ and that $1+\alpha-\alpha^2-\alpha^3>0$ for $\alpha<-1$.

Now suppose that $m\ge1$ and that the assertion has been established for all $k<m$.  Then, taking $k:=m-1$,
\begin{align*}
R^\flat_{2n+4m+1}(\alpha)&=R^\flat_{2n+4k+1}(\alpha)-\alpha^{2n+4k+2}+\alpha^{2n+4k+3}+\alpha^{2n+4k+4}-\alpha^{2n+4k+5}\\
&=R^\flat_{2n+4k+1}(\alpha)-\alpha^{2n+4k+2}(1-\alpha-\alpha^2+\alpha^3)>0,
\end{align*}
where we have used the induction hypothesis and the fact that $1-\alpha-\alpha^2+\alpha^3<0$ for $\alpha<-1$.
It follows that $\rho^\flat_{2n+4m+2}(\alpha)=-1$. Therefore, letting again $k=m-1$,
\begin{align*}
R^\flat_{2n+4m+2}(\alpha)&=R^\flat_{2n+4k+2}(\alpha)+\alpha^{2n+4k+3}+\alpha^{2n+4k+4}-\alpha^{2n+4k+5}-\alpha^{2n+4k+6}\\
&=R^\flat_{2n+4k+2}(\alpha)+\alpha^{2n+4k+3}(1+\alpha-\alpha^2-\alpha^3)<0.
\end{align*}
It follows that $\rho^\flat_{2n+4m+3}(\alpha)=+1$, and so $R^\flat_{2n+4m+3}(\alpha)=R^\flat_{2n+4m+2}(\alpha)+\alpha^{2n+4m+3}<0$. Finally, 
\begin{align*}
R^\flat_{2n+4m+4}(\alpha)&=R^\flat_{2n+4(m-1)+4}(\alpha)-\alpha^{2n+4(m-1)+5}-\alpha^{2n+4m+2}+\alpha^{2n+4m+3}+\alpha^{2n+4m+4}\\
&=R^\flat_{2n+4(m-1)+4}(\alpha)-\alpha^{2n+4m+1}(1+\alpha-\alpha^2-\alpha^3)>0.
\end{align*}
This concludes the proof.
\end{proof}

\begin{proof}[Proof of Theorem~\ref{alpha negative thm}] Let
$$\cZ(\alpha):=\Big\{n\in\bN_0\,\Big|\,R_n^\sharp(\alpha)=0\Big\}.
$$
Then Lemmas~\ref{alpha negative first rho entries lemma} and~\ref{alpha negative sequence lemma} imply that $|\cZ(\alpha)|=1$ if $\alpha\in\{x_1,x_2,\dots\}$ and $|\cZ(\alpha)|=0$ otherwise. Therefore, Proposition~\ref{Takagi class number of max prop} yields that $f_\alpha$ will have two maximizers in $[0,1/2]$ in the first case and one in the second case.  We will show next that these maximizers are given by the numbers $t_n$. Since those numbers are all different from $1/2$, the assertion on the number of maximizers in $[0,1]$ will follow. 

Next, Lemma~\ref{alpha negative sequence lemma} implies that for $m,n\in\bN_0$,
\begin{equation*}
\begin{split}
\rho^\flat_{2n+4m+2}=-1,\ \rho^\flat_{2n+4m+3}=+1,\ \rho^\flat_{2n+4m+4}=+1,\ \rho^\flat_{2n+4m+5}=-1\qquad\text{on $[x_n,x_{n+1})$,}\\
\rho^\sharp_{2n+4m+2}=-1,\ \rho^\sharp_{2n+4m+3}=+1,\ \rho^\sharp_{2n+4m+4}=+1,\ \rho^\sharp_{2n+4m+5}=-1\qquad\text{on $(x_n,x_{n+1}]$.}
\end{split}
\end{equation*}

With 
Lemma~\ref{alpha negative first rho entries lemma} we hence obtain  that for $\alpha\in[x_n,x_{n+1})$,
\begin{align*}
\tau^\flat(\alpha)=\sum_{m=1}^{2n+1}2^{-(m+1)}+\sum_{m=0}^\infty\Big(2^{-(2n+4m+3)}+2^{-(2n+4m+6)}\Big)=\frac{5-4^{-n}}{10}.
\end{align*}
Finally, Lemmas~\ref{alpha negative first rho entries lemma}  and~\ref{alpha negative sequence lemma} also give that $\tau^\sharp(\alpha)=\tau^\flat(\alpha-)$ for all $\alpha$ and this identifies the maximum location(s).

To identify the value of the maximum,  we need to compute $f_\alpha(t_n)$. The periodicity of $\phi$ implies that
\begin{align*}
f_\alpha(t_n)=\sum_{m=0}^\infty\frac{\alpha^m}{2^m}\phi\Big(2^m\frac{5-4^{-n}}{10}\Big)=\frac{5-4^{-n}}{10}+\sum_{m=1}^\infty\frac{\alpha^m}{2^m}\phi\Big(\frac{2^{m-2n}}{10}\Big).
\end{align*}
If $m\le 2n+2$, then $\frac{2^{m-2n}}{10}<1/2$, and so $\phi(\frac{2^{m-2n}}{10})=\frac{2^{m-2n}}{10}$. It follows that 
$$\sum_{m=1}^{2n+2}\frac{\alpha^m}{2^m}\phi\Big(\frac{2^{m-2n}}{10}\Big)=\frac{2^{-2n}}{10}\cdot\frac{\alpha(\alpha^{2n+2}-1)}{\alpha-1}.
$$
If $m\ge 2n+3$, then $\phi(\frac{2^{m-2n}}{10})=1/5$ if $m$ is odd and $\phi(\frac{2^{m-2n}}{10})=2/5$ if $m$ is even. It follows that 
\begin{align*}
\sum_{m=1}^\infty\frac{\alpha^m}{2^m}\phi\Big(\frac{2^{m-2n}}{10}\Big)&=\frac{\alpha^{2n+3}}{2^{2n+3}}\sum_{k=0}^\infty\frac{1+\alpha}5\cdot\Big(\frac\alpha2\Big)^{2k}
=\frac{\alpha^{2n+3}(1+\alpha)}{5\cdot 2^{2n+3}(1-(\alpha/2)^2)}.
\end{align*}
Putting everything together and simplifying yields the assertion.
\end{proof}

\subsection{Proof of result from Section~\ref{(1/2,1] section}
}

We start with a lemma that will have several applications in the proofs of this section.

\begin{lemma}\label{dense set lemma 1}
	Suppose that $\bm\rho$ satisfies $\rho_0=1$ and the step condition for  $\alpha \in (1/2,1)$. Then, for any $n \in \bN_0$, there exists  $n_0 > n$ such that $R_n(\alpha) := \sum_{m =0}^{n} \rho_m \alpha^m$ satisfies
	\begin{equation}\label{dense lemma condition eq}
		R_{n_0}(\alpha)R_{n_0+1}(\alpha) \leq 0.
	\end{equation}
\end{lemma}

\begin{proof}If the maximizer of $f_\alpha$ is not unique, then Theorem~\ref{Cantor thm} implies that there exists $n_0\in\bN$ such that $R_{kn_0}(\alpha)=0$ for all $k\in\bN$. Hence, the assertion is obvious in this case. Otherwise, we have $\bm\rho=\bm\rho^\flat=\bm\rho^\sharp$. Observe that $1 - \sum_{m = 1}^{\infty}\alpha^m=(1-2\alpha)/(1 - \alpha) < 0$, as $\alpha\in(1/2,1)$. Therefore, there exists  $n_0>0$ such that 
	\begin{equation}\label{dense lemma construction eq1}
	1 - \sum_{m = 1}^{k}\alpha^m \geq 0\quad\text{for all $k \leq n_0$} \quad \text{and} 	\quad 
	1 - \sum_{m = 1}^{n_0+1}\alpha^m < 0.	\end{equation}
	It follows that we must have $\rho_1=\cdots=\rho_{n_0+1}=-1$ and $\rho_{n_0+2}=+1$. Moreover, the inequalities \eqref{dense lemma condition eq}
 and on the left-hand side of \eqref{dense lemma construction eq1}
 must be strict.
	Therefore, we have that  $R_{n_0+1}(\alpha)R_{n_0}(\alpha) \leq 0$. This establishes the assertion for $n = 0$.
	
	For general $n$, we proceed by induction. So let us  suppose that $n \geq 1$ and that the assertion has been established for all $m \leq n-1$. By induction hypothesis, there exists $n_0> n-1$ such that \eqref{dense lemma condition eq} holds. If $n_0>n$, we are done. So we only need to consider the case $n_0=n$. Then $R_n(\alpha)R_{n+1}(\alpha)<0$. 
	If $R_{n}(\alpha)>0$, then $\rho_{n+1}=-1$ and hence $0 > R_{n+1}(\alpha) = R_{n}(\alpha) - \alpha^{n+1}> -\alpha^{n+1}$. In turn we get $\rho_{n+2}=+1$. Moreover, since $\alpha\in(1/2,1)$,
	$$R_{n+1}(\alpha)+\sum_{m=n+2}^\infty\alpha^m>  -\alpha^{n+1}+\frac{\alpha^{n+2}}{1-\alpha}=\frac{\alpha^{n+1}(2\alpha-1)}{1-\alpha}	>0.$$
	Therefore, the assertion follows as in the case $n=0$. If $R_{n}(\alpha)<0$, then we can use the same argument with switched signs. \end{proof}
	
	The first application of the preceding lemma concerns the possibility of $t=1/2$ being a maximizer of $f_\alpha$. As we saw in Proposition~\ref{alpha in -1 1/2 prop}, this is what happens for $\alpha\in[-1,1/2]$. The following result is also contained in~\cite[Theorem 4]{GalkinGalkina}, but we can give a very short proof here. 
	
	\begin{lemma}\label{t=1/2 lemma}The value $t=1/2$ is not a maximizer of $f_\alpha$ if  $\alpha>1/2$.	\end{lemma}

	\begin{proof}Note that $t=1/2$ has the Rademacher expansion $\bm\rho=(+1,-1,-1,\dots)$. Thus, if $t=1/2$  were a maximizer, then $\bm\rho$ would have to satisfy the step condition by Theorem~\ref{general Takagi step cond thm}. But this would contradict Lemma~\ref{dense set lemma 1}. 
	\end{proof}

\begin{proof}[Proof of Theorem~\ref{Cantor thm}] In case (a), Proposition~\ref{Takagi class number of max prop} yields that $f_\alpha$ has a unique maximizer in $[0,1/2]$. By Lemma~\ref{t=1/2 lemma}, this maximizer is strictly smaller than $1/2$. Therefore, there are exactly two maximizers in $[0,1]$.

In case (b), it is easy to see that a $\{-1,+1\}$-valued sequence satisfies the step condition for $\alpha$ if and only if it is made up of successive blocks of the form $\rho^\sharp_0,\dots, \rho^\sharp_{n_0}$ or $(-\rho^\sharp_0),\dots,(- \rho^\sharp_{n_0})$. Hence, Theorem~\ref{general Takagi step cond thm}
 identifies precisely those sequences as the Rademacher expansions of the minimizers of $f_\alpha$. Finally, Lemma~\ref{Hausdorff lemma} yields the assertion on the Hausdorff dimension and the Hausdorff measure.
\end{proof}

 \begin{proof}[Proof of Corollary~\ref{dyadic cor}] Let us suppose by way of contraction that $f_\alpha$ has a maximizer of the form $k2^{-n}$ for some $n\in\bN$ and $k\in\{0,\dots, 2^n\}$. 
 By Lemma~\ref{t=1/2 lemma}, we cannot have $t=1/2$.  It is moreover clear that the cases $t=0$ and $t=1$ are impossible. By symmetry of $f_\alpha$, we may thus assume that $t\in(0,1/2)$. Since $t$ is a dyadic rational number, it will have two distinct Rademacher expansions $\bm\rho$ and $\wt{\bm\rho}$ with $\rho_0=\wt\rho_0=1$. Moreover, there will be $n_1\in\bN$ such that one of them, say $\bm\rho$, satisfies $\rho_n=+1$ for $n\ge n_1$, whereas $\wt\rho_n=-1$ for $n\ge n_1$.
 By Theorem~\ref{general Takagi step cond thm}, both $\bm\rho$ and $\wt{\bm\rho}$ satisfy the step condition. Hence, Proposition~\ref{Takagi class number of max prop} implies that there exists a minimal $n_0\in\bN$ such that $\sum_{m=0}^{n_0}\alpha^m\rho_m=0$. By Theorem~\ref{Cantor thm}, both $\bm\rho$ and $\wt{\bm\rho}$ must therefore be formed out of blocks of the form $\rho_0,\dots,\rho_{n_0}$ or $(-\rho_0),\dots,(-\rho_{n_0})$. But then  these two blocks must be equal to $1,\dots,1$ and $-1,\dots, -1$, and every sequence formed by these blocks must be a maximizer. This implies that $0$ and $1$ are maximizers, which is impossible. \end{proof}

\begin{lemma}\label{dense set lemma 2}
	In the context of Lemma~\ref{dense set lemma 1}, we have $R_n(\alpha) \lra 0$ as $n\ua\infty$.
\end{lemma}
\begin{proof}
	For any $n \in \bN$, we have that $R_{n+1}(\alpha) = R_n(\alpha) + \rho_{n+1}\alpha^{n+1}$. Hence, if $R_{n+1}(\alpha)R_n(\alpha) \leq 0$, then we  must have that $|R_{n+1}(\alpha)| \leq \alpha^{n+1}$. Otherwise, the fact that $\rho_{n+1}R_n(\alpha)\le0$ implies that
	\begin{equation*}
		|R_{n+1}(\alpha)| = |R_{n}(\alpha) + \rho_{n+1} \alpha^{n+1}| \leq |R_n(\alpha)|.
	\end{equation*}
	Combining these two inequalities and using Lemma~\ref{dense set lemma 1} yields that for any $n \in \mathbb{N}$, there exists $n_0 > n$, such that for all $m > n_0$ we have 	 $|R_m(\alpha)|  \leq \alpha^{n_0}$. This proves the assertion.
\end{proof}

\begin{lemma}\label{dense set lemma 3}
	For  $\alpha \in (1/2,1)$ and every $\eps > 0$, there exists $\beta \in (\alpha - \eps,\alpha+\eps)\cap(1/2,1)$ such that $\bm\rho^\sharp(\alpha) \neq \bm\rho^\sharp(\beta)$.
\end{lemma}
\begin{proof}
	Let us assume by way of contradiction that there exists $\alpha \in (1/2,1)$ and $\eps > 0$ that $\bm\rho:=\bm\rho^\sharp(\alpha) = \bm\rho^\sharp(\beta)$ for all $\beta \in (\alpha - \eps,\alpha + \eps)$. 
	Since
	$
		\limsup_{n} \sqrt[n]{|\rho_n|} = 1
	$,
	 $R(z):= \sum_{m =0}^\infty \rho_mz^m$ is an analytic function of $z\in(-1,1)$. Take $\eps>0$ so that  $(\alpha -\eps,\alpha + \eps) \subset (-1,1)$. Then  Lemma~\ref{dense set lemma 2} implies that $R(z) = 0$ for all $z \in (\alpha -\eps,\alpha + \eps)$ and in turn $R(z) = 0$ for all $z \in (-1,1)$. But this implies $\rho_n=0$ for all $n$ and hence a contradiction. 
\end{proof}

\begin{proof}[Proof of Theorem~\ref{nowhere flat thm}] We prove the assertion only for $\tau^\sharp$; the proof for $\tau^\flat$ is identical. Let $\alpha\in(1/2,1) $  and  $\eps > 0$ be given. By Lemma~\ref{dense set lemma 3} there exists $\beta \in (\alpha -\eps,\alpha +\eps)\cap(1/2,1)$, such that $\bm\rho^\sharp(\alpha) \neq \bm\rho^\sharp(\beta)$. By Corollary~\ref{dyadic cor}, neither $\bm\rho^\sharp(\alpha) $ nor $\bm\rho^\sharp(\beta) $ can be a Rademacher expansion of a dyadic rational number. Therefore, we must have 
	$
		\tau^\sharp(\alpha) = \cT(\bm\rho^\sharp(\alpha)) \neq \cT(\bm\rho^\sharp(\beta)) = \tau^\sharp(\beta)
	$. Now suppose by way of contradiction that there are $\alpha\in(1/2,1) $  and  $\eps > 0$ such that $\tau^\sharp$ is continuous on $(\alpha -\eps,\alpha +\eps)\cap(1/2,1)$. Let $\beta\in (\alpha -\eps,\alpha +\eps)\cap(1/2,1)$ be as above. By the intermediate value theorem, the continuous function $\tau^\sharp$ would have to take every value between $\tau^\sharp(\alpha)$ and $\tau^\sharp(\beta)$, but this contradicts Corollary~\ref{dyadic cor}.
	 \end{proof}

\subsection{Proof of result from Section~\ref{minima section}
}

\begin{proof}[Proof of Theorem~\ref{minima thm}] As discussed in 
Remark~\ref{minimizer rem}, we define $\bm\lambda^\flat(\alpha)$ and $\bm\lambda^\sharp(\alpha)$ by $\lambda^\flat_0(\alpha)=\lambda^\sharp_0(\alpha)=+1$ and 
\begin{equation*}
\lambda_n^\flat(\alpha)=\begin{cases}+1&\text{if $\sum_{m=0}^{n-1}\alpha^m\lambda^\flat_m(\alpha)>0$,}\\
-1&\text{otherwise,}
\end{cases}\qquad \lambda_n^\sharp(\alpha)=\begin{cases}+1&\text{if $\sum_{m=0}^{n-1}\alpha^m\lambda^\sharp_m(\alpha)\ge0$,}\\
-1&\text{otherwise.}
\end{cases}
\end{equation*}
Then $\cT(\bm\lambda^\flat(\alpha))$ is the largest and  $\cT(\bm\lambda^\sharp(\alpha))$ is the smallest minimizer of $f_\alpha$ in $[0,1/2]$. For simplicity, we will suppress the argument $\alpha$ in this proof. We also let $L^\flat_n:=\sum_{m=0}^{n}\alpha^m\lambda^\flat_m$ and $L^\sharp_n:=\sum_{m=0}^{n}\alpha^m\lambda^\sharp_m$.

\ref{minima thm (a)} We prove by induction on $n\in\bN_0$ that both $\bm\lambda=\bm\lambda^\flat$ and $\bm\lambda=\bm\lambda^\sharp $ satisfy
\begin{equation}\label{minima thm (a) claim}
\lambda_{4n}=+1,\quad\lambda_{4n+1}=+1,\quad\lambda_{4n+2}=-1,\quad\lambda_{4n+3}=-1.
\end{equation}
Then  $f_\alpha$ will have a unique minimizer on $[0,1/2]$, which will be equal to 
$$\cT(\bm\lambda)=\sum_{n=0}^\infty(1-\lambda_n)2^{-n-1}=\sum_{n=0}^\infty2^{-4n}(2^{-3}+2^{-4})=\frac15.
$$
To prove \eqref{minima thm (a) claim}, consider first the case $n=0$. Then $L_0=\lambda_0=+1$ and so $\lambda_1=+1$. Hence $L_1=1+\alpha<0$ and thus $\lambda_2=-1$. Finally, $L_2=L_1-\alpha^2<0$, so that $\lambda_3=-1$. Now suppose that $n\ge1$ and the assertion has been proved for all $m<n$. Then 
$$L_{4n-1}=\sum_{k=0}^{n-1}\alpha^{4k}(1+\alpha-\alpha^2-\alpha^3)
=(1+\alpha)^2(1-\alpha)\frac{1-\alpha^{4k}}{1-\alpha^4}>0.$$
Hence $\lambda_{4n}=+1$. It follows that $L_{4n+1}=L_{4n}+\alpha^{4n}>0$ and in turn $\lambda_{4n+1}>0$. Therefore,
\begin{align*}
L_{4n+1}=1+\alpha-\sum_{k=0}^{n-1}\alpha^{4k+2}(1+\alpha-\alpha^2-\alpha^3)=1+\alpha-\alpha^2(1+\alpha)^2(1-\alpha)\frac{1-\alpha^{4k}}{1-\alpha^4}<0.
\end{align*}
Hence $\lambda_{4n+2}=-1$ and so $L_{n+3}=L_{n+2}-\alpha^{4n+2}<0$. Thus, we finally get $\lambda_{4n+3}=-1$.

To prove our formula for the minimum value,  recall from the proof of Theorem~\ref{alpha negative thm} that $\phi(2^m/5)=1/5$ for $m$ even and  $\phi(2^m/5)=2/5$ for $m$ odd. Hence,
\begin{align*}
f_\alpha(1/5)=\sum_{m=0}^\infty\frac{\alpha^m}{2^m}\phi(2^m/5)=\sum_{m=0}^\infty\frac{\alpha^{2m}}{2^{2m}}\Big(\frac15+\frac{\alpha}2\cdot\frac{2}{5}\Big)=\frac{1+\alpha}{5(1-(\alpha/2)^2)}.
\end{align*}

\ref{minima thm (b)} Suppose that $\bm\lambda$ is any sequence satisfying the step condition for minima, $\lambda_nL_{n-1}\ge0$. Then we have $\lambda_1L_0=\lambda_1\lambda_0\ge0$ and hence $\lambda_1=\lambda_0$. Moreover, we have $L_1=\lambda_0-\lambda_1=0$. From here, it follows from a straightforward induction argument that we must have $\lambda_{2n}=\lambda_{2n+1}$ for all $n\in\bN_0$ and that, conversely, any such sequence satisfies the step condition for minima. Hence, Remark~\ref{minimizer rem} in conjunction with Theorem~\ref{general Takagi step cond thm}
 yields that the set of minimizers of $f_{-1}$ is equal to the set of all those $t\in[0,1]$ whose binary expansion, $t=0.\eps_0\eps_1\cdots$ satisfies $\eps_{2n}=\eps_{2n+1}$ for $n\in\bN_0$. Since this set contains $t=0$, the minimum value of $f_{-1}$ must be $f_{-1}(0)=0$. Clearly, the set of minimizers can also be represented as the set of those $t\in[0,1]$ whose binary expansion is formed of successive blocks of  the digits $11$ and $00$. Therefore, the claim on its Hausdorff dimension follows from Lemma~\ref{Hausdorff lemma}.

\ref{minima thm (c)} Let $\alpha\in(-1,2)$ be given. We show by induction on $n$ that $\lambda_n^\flat=+1$ for all $n\in\bN_0$. 
For $n=0$, we clearly have $\lambda_0^\flat=1$. Now suppose that the claim has been established for all $m\le n$. Then 
\begin{equation*}
L_n^\flat =\sum_{m=0}^n\lambda_m^\flat \alpha^m=\sum_{m=0}^n\alpha^m=\frac{1-\alpha^{n+1}}{1-\alpha}>0,
\end{equation*}
which gives  $\lambda_{n+1}^\flat  =+1$. It follows that $\cT(\bm\lambda^\flat )=0$. Since $\cT(\bm\lambda^\flat )$ is the largest minimizer in $[0,1/2]$. the result follows.\end{proof}

\bibliography{CTbook}{}
\bibliographystyle{abbrv}

\vfill\eject

 \end{document}